\documentclass[11pt]{amsart}

\usepackage[T1]{fontenc}

\usepackage[utf8]{inputenc}
\usepackage{lmodern}
\usepackage[english]{babel}
\usepackage[autostyle]{csquotes}

\usepackage[pdfusetitle,linktocpage,colorlinks=false]{hyperref}
\usepackage{doi}
\usepackage{filecontents}
\usepackage{enumerate}

\usepackage{amsmath,amssymb,amsthm,amscd}
\usepackage{pictexwd,dcpic}
\newtheorem{thm}{Theorem}[section]
\newtheorem{prop}[thm]{Proposition}
\newtheorem{lem}[thm]{Lemma}
\newtheorem{rem}[thm]{Remark}
\newtheorem{cor}[thm]{Corollary}

\newtheorem{defn}[thm]{Definition}
\newtheorem{que}[thm]{Question}
\usepackage[colorinlistoftodos]{todonotes}

\newtheorem{thmx}{Theorem}
\newtheorem{corx}[thmx]{Corollary}

\AtEndDocument{\bigskip{\footnotesize%
  \textsc{Institute of Mathematics of the Polish Academy of Sciences, \'{S}niadeckich 8, 00-656 Warsaw, Poland} \par  
  \textit{E-mail address}: \texttt{ychung@impan.pl} \par
  \addvspace{\medskipamount}
  \textsc{Institute of Mathematics of the Polish Academy of Sciences, \'{S}niadeckich 8, 00-656 Warsaw, Poland} \par  
  \textit{E-mail address}: \texttt{kli@impan.pl} \par
}}

\begin{document}
\title[Structure and $K$-theory of $\ell^p$ uniform Roe algebras]{Structure and $K$-theory of $\ell^p$ uniform Roe algebras}
\author{Yeong Chyuan Chung \and Kang Li}

\thanks{Both authors are supported by the European Research Council (ERC-677120).}

\date{\today}
\maketitle

\begin{abstract}
In this paper, we characterize when the $\ell^p$ uniform Roe algebra of a metric space with bounded geometry is (stably) finite and when it is properly infinite in standard form for $p\in [1,\infty)$. Moreover, we show that the $\ell^p$ uniform Roe algebra is a (non-sequential) spatial $L^p$ AF algebra in the sense of Phillips and Viola if and only if the underlying metric space has asymptotic dimension zero. 

We also consider the ordered $K_0$ groups of $\ell^p$ uniform Roe algebras for metric spaces with low asymptotic dimension, showing that (1) the ordered $K_0$ group is trivial when the metric space is non-amenable and has asymptotic dimension at most one, and (2) when the metric space is a countable locally finite group, the (ordered) $K_0$ group is a complete invariant for the (bijective) coarse equivalence class of the underlying locally finite group. It happens that in both cases the ordered $K_0$ group does not depend on $p\in [1,\infty)$.
\end{abstract}
\parskip 8pt

\tableofcontents

\section{Introduction}
$\ell^2$ uniform Roe algebras are $C^*$-algebras associated with discrete bounded geometry metric spaces, and which reflect coarse geometric properties of the underlying metric spaces. These algebras have been well-studied and provide a link between coarse geometry of metric spaces and the structure theory of $C^*$-algebras (e.g., \cite{amenUR, MR1876896, MR1739727, MR3158244, LL, lowdimUR, MR1763912, MR2873171, MR3784048, MR1905840, Wei, WZ}).

In recent years, there has been an uptick in interest in the $\ell^p$ version of uniform Roe algebras for $p\in [1,\infty)$ from both the operator theory community and the coarse geometry community (e.g., \cite{CL, MR3451966, MR3926163, MR3212726, MR3151282, vspakula2017metric, vspakula2018quasi, Zhang18}). Moreover, $\ell^p$ uniform Roe algebras belong to the class of $L^p$ operator algebras in the sense of N. Christopher Phillips, and the structure and $K$-theory of certain $L^p$ operator algebras have been studied (e.g., \cite{phible, Chung2, Chung1, PhilOpen, Phil12, Phil13, PhilSimplicity, PhilViola}). Thus it is natural to study the structure and $K$-theory for $\ell^p$ uniform Roe algebras of metric spaces with bounded geometry.
\begin{defn}
A metric space $(X,d)$ has bounded geometry if for any $R>0$, there exists $N_R>0$ such that $|B_R(x)|\leq N_R$ for all $x\in X$, where $B_R(x)$ denotes the closed ball of radius $R$ centered at $x$.
\end{defn}

Let $(X,d)$ be a metric space with bounded geometry so that $(X,d)$ is necessarily countable and discrete. For a bounded operator $T=(T_{xy})_{x,y\in X}\in B(\ell^p(X))$, where $T_{xy}=(T\delta_y)(x)$, we define the propagation of $T$ to be
\[ \mathrm{prop}(T):=\sup\{ d(x,y):x,y\in X,T_{xy}\neq 0 \}\in[0,\infty]. \]

\begin{defn}
Let $(X,d)$ be a metric space with bounded geometry. For $1\leq p<\infty$, the associated $\ell^p$ uniform Roe algebra, denoted by $B^p_u(X)$, is defined to be the operator norm closure of the algebra of all bounded operators on $\ell^p(X)$ with finite propagation.
\end{defn}
\parskip 6pt

In Section 2, we investigate when the $\ell^p$ uniform Roe algebra $B^p_u(X)$ of a metric space $X$ with bounded geometry is (stably) finite and when it is properly infinite in standard form as a unital Banach algebra (as defined in \cite[Definition 1.1]{Laustsen} and \cite[Definition 1.1]{DLR}).

The (stable) finiteness of $B^p_u(X)$ can be characterized in terms of quasidiagonal set of operators on $\ell^p(X)$ in the sense of Halmos (see \cite{MR0270173}) and coarse connected components of the metric space $X$. Recall that two elements $x$ and $y$ in a metric space $(X,d)$ are said to be $R$-connected for some $R>0$ if there is a finite sequence $x_0,\ldots,x_n$ in $X$ with $x_0=x$, $x_n=y$, and $d(x_i,x_{i+1})\leq R$ for $i=0,\ldots,n-1$. This is an equivalence relation on $X$, and the equivalence classes are called the $R$-connected components of $X$. We are able to show the following theorem, which extends \cite[Theorem 2.3]{Wei} from the $p=2$ case to any $p\in[1,\infty)$.
\begin{thmx}[see Theorem \ref{thmFinite}]
Let $(X,d)$ be a metric space with bounded geometry, and let $p\in[1,\infty)$. The following are equivalent:
\begin{enumerate}
\item $B^p_u(X)$ is a quasidiagonal set in $B(\ell^p(X))$.
\item $B^p_u(X)$ is stably finite.
\item $B^p_u(X)$ is finite.
\item For each $R>0$, every $R$-connected component of $X$ is finite.

\end{enumerate}
\end{thmx} 
For the proper infiniteness in standard form of $B^p_u(X)$, we give many different characterizations involving $L^p$ Cuntz algebras $\mathcal{O}^p_2$ introduced by Phillips in \cite{Phil12} (see Definition \ref{Op_d}), the algebraic $K_0$ group of $B^p_u(X)$, normalized traces on $B^p_u(X)$, and paradoxical decompositions of $X$. The following theorem is an analog of \cite[Theorem 4.9]{amenUR}, which only deals with the $p=2$ case:

\begin{thmx}[see Theorem \ref{nonamenThm}]
Let $(X,d)$ be a metric space with bounded geometry, and let $p\in[1,\infty)$. The following are equivalent:
\begin{enumerate}
\item $B^p_u(X)$ is properly infinite in standard form: there exists an idempotent $e$ in $B^p_u(X)$ such that $e\sim 1\sim 1-e$, where $\sim$ denotes algebraic equivalence.

\item $X$ admits a paradoxical decomposition: there is a partition $X=X_+\sqcup X_-$ such that there exist two bijections $t_\pm:X\rightarrow X_\pm$ satisfying $\sup_{x\in X}d(x,t_\pm(x))<\infty$.
\item There is a unital isometric embedding of $\mathcal{O}^p_2$ into $B^p_u(X)$.
\item $[1]_0=[0]_0$ in the algebraic $K_0$ group $K_0(B^p_u(X))$.
\item $B^p_u(X)$ has no normalized trace.
\item There is no unital linear functional $\psi$ of norm one on $B(\ell^p(X))$ such that $\psi(aT)=\psi(Ta)$ for all $T\in B(\ell^p(X))$ and all $a\in B^p_u(X)$.
\end{enumerate}
\end{thmx}

\parskip 6pt

In Section 3, we study the group homomorphism $K_0(\ell^\infty(X))\rightarrow K_0(B^p_u(X))$ induced by the canonical diagonal inclusion $\ell^\infty(X)\hookrightarrow B^p_u(X)$ for every metric space $X$ with low asymptotic dimension. 
\begin{defn}
A collection $\mathcal{U}$ of subsets of a metric space is $r$-separated if for all distinct $A,B\in\mathcal{U}$, we have $d(A,B)>r$. A metric space with bounded geometry has asymptotic dimension at most $d$ if for any $r>0$, there is a decomposition $X=U_0\sqcup\cdots\sqcup U_d$, where each $U_i$ in turn decomposes into a uniformly bounded collection of $r$-separated subsets.
\end{defn}

The main result of Section 3 is the following theorem:
\begin{thmx} [see Theorem \ref{thm:lowdim}] 
Let $(X,d)$ be a metric space with bounded geometry. If the asymptotic dimension of $X$ is at most one, then the group homomorphism $K_0(\ell^\infty(X))\rightarrow K_0(B^p_u(X))$ induced by the canonical diagonal inclusion is always surjective for all $p\in[1,\infty)$.
\end{thmx}

As a consequence, we deduce the following result:
\begin{corx} [see Corollary \ref{non-amen ass one}] 
If $X$ is a non-amenable metric space with bounded geometry and it has asymptotic dimension at most one, then $K_0(B^p_u(X))=0$ for all $p\in[1,\infty)$.
\end{corx}
In order to prove Theorem C, we have to use a quantitative version of $K$-theory for $L^p$ operator algebras from \cite{Chung1}. The idea of the proof is almost identical to the one in \cite[Section 5]{lowdimUR} except that the quantitative version of Bott periodicity was not proved in this setting and we consider suspensions of $\ell^p$ uniform Roe algebras instead.

\parskip 6pt
In Section 4, we study structure and $K$-theory of $\ell^p$ uniform Roe algebras of metric spaces with zero asymptotic dimension. We begin with the following characterization of spaces with zero asymptotic dimension in terms of their $\ell^p$ uniform Roe algebras:
\begin{thmx} [see Theorem \ref{thmAsdim0}] 
Let $X$ be a metric space with bounded geometry, and let $p\in[1,\infty)$. The following are equivalent:
\begin{enumerate}
\item $X$ has asymptotic dimension zero.
\item $B^p_u(X)$ is an inductive limit of 
$\bigoplus_{k=1}^N M_{d_k}^p$, where $N,d_1,\ldots,d_k\in\mathbb{N}$ and $M_d^p$ denotes $B(\ell^p(\{1,\ldots,d\}))$.

\item $B^p_u(X)$ has cancellation (see Definition \ref{cancellation}).
\end{enumerate}
\end{thmx}

In the remainder of Section 4, we restrict our attention to the case where the metric space is a countable, locally finite group equipped with a proper left-invariant metric. Such a metric group is actually a bounded geometry metric space with asymptotic dimension zero. We compute the ordered $K_0$ group $(K_0(B^p_u(\Gamma)),K_0(B^p_u(\Gamma))^+,[1]_0)$ for any countable, locally finite group $\Gamma$, showing that it is independent of $p$ (see Theorem \ref{k-theory of locally finite}). As a consequence, the (ordered) $K_0$ group of the associated $\ell^p$ uniform Roe algebra is a complete invariant for the (bijective) coarse equivalence class of the underlying countable locally finite group.

\begin{defn}
Let $X$ and $Y$ be metric spaces. 
\begin{itemize}
\item A (not necessarily continuous) map $f:X\rightarrow Y$ is said to be \emph{uniformly expansive} if for all $R>0$ there exists $S>0$ such that if $x_1,x_2\in X$ satisfy $d(x_1,x_2)\leq R$, then $d(f(x_1),f(x_2))\leq S$.

\item Two maps $f,g:X\rightarrow Y$ are said to be \emph{close} if there exists $C>0$ such that $d(f(x),g(x))\leq C$ for all $x\in X$.

\item Two metric spaces $X$ and $Y$ are said to be \emph{coarsely equivalent} if there exist uniformly expansive maps $f:X\rightarrow Y$ and $g:Y\rightarrow X$ such that $f\circ g$ and $g\circ f$ are close to the identity maps, respectively. In this case, we say both $f$ and $g$ are \emph{coarse equivalences} between $X$ and $Y$.

\item We say a map $f:X\to Y$ is a \emph{bijective coarse equivalence} if $f$ is both a coarse equivalence and a bijection. In this case, we say $X$ and $Y$ are \emph{bijectively coarsely equivalent}.

\end{itemize}
\end{defn}

\begin{corx}[see Corollary \ref{ce+lf}]
Let $\Gamma$ and $\Lambda$ be countable, locally finite groups with proper left-invariant metrics $d_\Gamma$ and $d_\Lambda$ respectively. Then the following conditions are equivalent:
\begin{enumerate}
\item $(\Gamma,d_\Gamma)$ and $(\Lambda,d_\Lambda)$ are coarsely equivalent. 
\item $K_0(B^p_u(\Gamma))\cong K_0(B^p_u(\Lambda))$ for all $p\in[1,\infty)$.
\item $K_0(B^p_u(\Gamma))\cong K_0(B^p_u(\Lambda))$ for some $p\in[1,\infty)$.
\end{enumerate}
\end{corx}

\begin{thmx}[see Theorem \ref{thm:locfingrp}]
Let $\Gamma$ and $\Lambda$ be countable, locally finite groups with proper left-invariant metrics $d_\Gamma$ and $d_\Lambda$ respectively. Then the following conditions are equivalent:
\begin{enumerate}
\item $(\Gamma,d_\Gamma)$ and $(\Lambda,d_\Lambda)$ are bijectively coarsely equivalent.
\item For every $p\in[1,\infty)$, there is an isometric isomorphism $\phi:B^p_u(\Gamma)\rightarrow B^p_u(\Lambda)$ such that $\phi(\ell^\infty(\Gamma))=\ell^\infty(\Lambda)$.

\item $B^p_u(\Gamma)$ and $B^p_u(\Lambda)$ are isometrically isomorphic for some $p\in[1,\infty)$.

\item $(K_0(B^p_u(\Gamma)),K_0(B^p_u(\Gamma))^+,[1]_0)\cong(K_0(B^p_u(\Lambda)),K_0(B^p_u(\Lambda))^+,[1]_0)$ for every $p\in[1,\infty)$.

\item $(K_0(B^p_u(\Gamma)),[1]_0)\cong(K_0(B^p_u(\Lambda)),[1]_0)$ for some $p\in[1,\infty)$.
\end{enumerate}
\end{thmx}

We end the introduction by asking the following question:
\begin{que}
Let $X$ be a metric space with bounded geometry. Does the (ordered) $K_0$ group of $B^p_u(X)$ depend on $p\in [1,\infty)$ when $X$ has finite asymptotic dimension? 
\end{que}
We have partial answers from Corollary~\ref{non-amen ass one} for non-amenable spaces with asymptotic dimension at most one and from Theorem~\ref{k-theory of locally finite} for countable locally finite groups.

\section{Finiteness and proper infiniteness of $\ell^p$ uniform Roe algebras}

In this section, we investigate when the $\ell^p$ uniform Roe algebra of a metric space with bounded geometry is (stably) finite and when it is properly infinite in standard form.

The following definition of a (stably) finite (resp. properly infinite) algebra is based on \cite[Definition 1.1]{Laustsen} and \cite[Definition 1.1]{DLR}.

\begin{defn}
We say that two idempotents $e$ and $f$ in an algebra $A$ are algebraically equivalent, and write $e\sim f$, if there exist $x,y\in A$ such that $e=xy$ and $f=yx$.

We say that $e$ and $f$ are orthogonal (and write $e\perp f$) if $ef=0=fe$. We write $e\leq f$ if $ef=fe=e$.

An idempotent $e$ in $A$ is said to be properly infinite if there are orthogonal idempotents $e_1$ and $e_2$ in $eAe$ such that $e_1\sim e\sim e_2$.

A (nonzero) unital algebra $A$ is said to be properly infinite if the unit $1_A$ is a properly infinite idempotent.

A unital, properly infinite algebra $A$ is said to be properly infinite in standard form if there exists an idempotent $e$ in $A$ such that $e\sim 1_A\sim 1_A-e$.

An idempotent $e$ in $A$ is said to be finite if whenever $f$ is an idempotent in $A$ such that $e\sim f$ and $f\leq e$, then $f=e$.

A unital algebra $A$ is said to be finite if the unit $1_A$ is a finite idempotent. If $M_n(A)$ is finite for all $n\in\mathbb{N}$, then we say that $A$ is stably finite. 
\end{defn}

\begin{rem}
There are Banach algebras that are properly infinite but not properly infinite in standard form. For example, the Cuntz algebras $\mathcal{O}_n$ ($2\leq n\leq\infty$) introduced in \cite{Cuntz77} are $C^\ast$-algebras that are properly infinite, but only $\mathcal{O}_2$ is properly infinite in standard form (see \cite{Cuntz78} or \cite{Cuntz81}).
\end{rem}

%
%

\subsection{Finiteness}

We will show that finiteness of the $\ell^p$ uniform Roe algebra $B^p_u(X)$ of a bounded geometry metric space $X$ is equivalent to $B^p_u(X)$ being stably finite and is also equivalent to it being a quasidiagonal set of operators on $\ell^p(X)$. It can also be characterized in terms of coarse connected components of $X$.

We begin with a lemma providing useful criteria for determining finiteness of a unital Banach algebra.

\begin{lem} \label{finite}
Let $A$ be a unital Banach algebra. The following are equivalent:
\begin{enumerate}
\item $A$ is finite.
\item Every left-invertible element in $A$ is invertible.
\item Every right-invertible element in $A$ is invertible.
\item All idempotents in $A$ are finite.
\end{enumerate}
\end{lem}

\begin{proof}
(i) $\Rightarrow$ (ii): Suppose that $ab=1_A$. Then $1_A=ab\sim ba\leq 1_A$ so $ba=1_A$.

(ii) $\Leftrightarrow$ (iii): It is straightforward.

(ii) $\Rightarrow$ (iv): Suppose that $e,f$ are idempotents in $A$ such that $e\sim f$ and $f\leq e$. There exist $x,y\in A$ such that $xy=e$ and $yx=f$. Moreover, we may assume that $x=xf$ and $y=fy$. Let $s=y+1_A-e$ and $t=x+1_A-e$.
Then $x(1_A-e)=0=(1_A-e)y$ so $ts=1_A$ and $st=f+1_A-e$. Since $s$ is left-invertible, and hence invertible, we get $e=f$.

(iv) $\Rightarrow$ (i): It is clear.
\end{proof}

The next definition is adapted from the case of operators on Hilbert space (cf. \cite[Definition 7.2.1]{BO}).

\begin{defn}
Let $E$ be a Banach space, and let $A\subseteq B(E)$ be an arbitrary collection of operators. We say that $A$ is a quasidiagonal set of operators on $E$ if for each finite set $M\subset A$, each finite set $F\subset E$, and each $\varepsilon>0$, there exists a finite rank idempotent operator $P$ on $E$ of norm one such that $||[S,P]||<\varepsilon$ and $||P\xi-\xi||<\varepsilon$ for all $S\in M$ and $\xi\in F$.

If $A=\{S\}$, then we say that the operator $S$ is quasidiagonal.
\end{defn}

\begin{prop} \label{QDprop}
A left-invertible quasidiagonal operator on a Banach space $E$ is invertible.
\end{prop}

\begin{proof}
Let $0<\varepsilon<\frac{1}{2}$, let $S$ be a left-invertible quasidiagonal operator on $E$, let $T\in B(E)$ be such that $TS=I$, and let $\xi\in E$ be a unit vector.
Let $P$ be a finite rank idempotent operator on $E$ of norm one such that $||[P,S]||<\varepsilon||T||^{-1}$, and $||P\xi-\xi||<\varepsilon$.
Then \[||PSP-SP||=||[P,S]P||<\varepsilon||T||^{-1}.\] Let $W=P-TPSP$. Then $WP=W$, $||W||=||T(SP-PSP)||<\varepsilon$ and \[(\sum_{n=0}^\infty W^nT)PSP=\sum_{n=0}^\infty W^n(P-W)=\lim_{n\rightarrow \infty}(P-W^{n+1})=P.\] It follows that $PSP$ is an injective operator on the finite-dimensional space $PE$, so it is surjective as an operator on $PE$.
Thus there exists $\eta\in PE$ such that $P\xi=PSP\eta$, and $\eta=\sum_{n=0}^\infty W^nTP\xi$.
Then 
\begin{align*}
||\xi-S\eta||&=||\xi-P\xi+PSP\eta-SP\eta|| \\
&< \varepsilon+||PSP\eta-SP\eta|| \\ &<\varepsilon+\varepsilon||T||^{-1}||\eta|| \\
&\leq \varepsilon+\varepsilon||T||^{-1}\sum_{n=0}^\infty||W||^n||T|| \\ &< 3\varepsilon.
\end{align*}

The arguments above show that every unit vector $\xi\in E$ is in the closure of the range of $S$. Since $S$ is bounded below, it follows that $S$ has the closed range. Hence, $S$ is invertible.
\end{proof}


One more ingredient for our result is the notion of a coarse connected component in a metric space. In a metric space $(X,d)$, two elements $x$ and $y$ are said to be $R$-connected for some $R>0$ if there is a finite sequence $x_0,\ldots,x_n$ in $X$ with $x_0=x$, $x_n=y$, and $d(x_i,x_{i+1})\leq R$ for $i=0,\ldots,n-1$. This is an equivalence relation on $X$, and the equivalence classes are called the $R$-connected components of $X$. For convenience, we also consider $R=0$, in which case the 0-connected components are just the points of $X$. We say that $X$ is coarsely connected if there exists $R>0$ such that every pair of points in $X$ is $R$-connected. 

Note that finitely generated groups are 1-connected with respect to any word metric. Moreover, a countable discrete group $\Gamma$ is finitely generated if and only if $\Gamma$ is coarsely connected with respect to any/some proper left-invariant metric \cite[Lemma 7.2]{MR2645049}. 

\begin{lem} \label{WeiLem2.4} \cite[Lemma 2.4]{Wei}
Let $(X,d)$ be an infinite metric space with bounded geometry. The following are equivalent:
\begin{enumerate}
\item For any $R>0$, there is no $R$-connected infinite sequence $\{x_n\}_{n=1}^\infty$ in $X$.
\item The space $X$ has no coarsely connected subspace containing infinitely many points.
\item There exists a sequence of disjoint (non-empty) finite subsets $\{X_n\}_{n=1}^\infty$ of $X$ such that $X=\bigsqcup_{n=1}^\infty X_n$ and $\lim_{n\rightarrow\infty} d(X_n,\bigsqcup_{i=1}^{n-1}X_i)=\infty$.
\end{enumerate}
\end{lem}

The following theorem is a generalization of \cite[Theorem 2.3]{Wei}, which deals with the $p=2$ case.

\begin{thm} \label{thmFinite}
Let $(X,d)$ be a metric space with bounded geometry, and let $p\in[1,\infty)$. The following are equivalent:
\begin{enumerate}
\item $B^p_u(X)$ is a quasidiagonal set in $B(\ell^p(X))$.
\item $B^p_u(X)$ is stably finite.
\item $B^p_u(X)$ is finite.
\item For each $R>0$, every $R$-connected component of $X$ is finite.

\end{enumerate}
\end{thm}

\begin{proof}
(i) $\Rightarrow$ (ii): 
%
%
If $B^p_u(X)$ is a quasidiagonal set in $B(\ell^p(X))$, then note that $M_n(B^p_u(X))$ is a quasidiagonal set in $B(\ell^p(X)^{\oplus n})$ for each $n\in\mathbb{N}$. It then follows from Lemma \ref{finite} and Proposition \ref{QDprop} that $M_n(B^p_u(X))$ is finite for each $n\in\mathbb{N}$.

(ii) $\Rightarrow$ (iii) is trivial.

(iii) $\Rightarrow$ (iv): The proof of this is essentially the same as that of \cite[Lemma 2.5]{Wei}.
Suppose there is an $R$-connected infinite sequence $\{x_n\}_{n=1}^\infty$ in $X$ for some $R>0$. Define an operator $S$ on $\ell^p(X)$ satisfying
\begin{align*}
S\delta_x&=\delta_x, x\in X\setminus\{x_n:n=1,2,\ldots\} \\
S\delta_{x_n}&=\delta_{x_{n+1}}, n=1,2,\ldots.
\end{align*}
Then $S$ is a non-invertible isometry on $\ell^p(X)$. Moreover, $S$ has finite propagation since $\{x_n\}_{n=1}^\infty$ is $R$-connected.
Thus $S$ is a non-invertible but left-invertible element in $B^p_u(X)$, so $B^p_u(X)$ is not finite by Lemma \ref{finite}.


(iv) $\Rightarrow$ (i): If $X$ is finite with $|X|=n$, then $B^p_u(X)=M_n(\mathbb{C})=B(\ell^p(X))$ and (i) is obvious.
If $X$ is infinite, then by Lemma \ref{WeiLem2.4}, there exists a sequence of disjoint finite subsets $\{X_n\}_{n=1}^\infty$ of $X$ such that $X=\bigsqcup_{n=1}^\infty X_n$ and $\lim_{n\rightarrow\infty} d(X_n,\bigsqcup_{i=1}^{n-1}X_i)=\infty$.

The remainder of the proof is the same as that of the implication [(f) $\Rightarrow$ (a)] in \cite[Theorem 2.3]{Wei} upon replacing $\ell^2(X)$ by $\ell^p(X)$, and we reproduce it here for the reader's convenience.

Since the set of finite propagation operators in $B^p_u(X)$ is dense in $B^p_u(X)$, it suffices to show that for every finite set $M\subset B^p_u(X)$ consisting of finite propagation operators, every finite set of vectors $F\subset\ell^p(X)$, and every $\varepsilon>0$, there exists a finite rank idempotent operator $P\in B(\ell^p(X))$ of norm one such that $||TP-PT||<\varepsilon$ and $||P\xi-\xi||<\varepsilon$ for all $T\in M$ and $\xi\in F$.

Since $\ell^p(X)=\bigoplus_n\ell^p(X_n)$, let $P_n$ be the idempotent operator $\ell^p(X)\rightarrow\bigoplus_{i=1}^n\ell^p(X_i)$. Then $P_n$ is a finite rank idempotent operator of norm one, and the increasing sequence $(P_n)$ converges strongly to the identity. Thus there exists $n_0$ such that $||P_n\xi-\xi||<\varepsilon$ for all $\xi\in F$ and $n\geq n_0$.
Since $M$ is a finite set of finite propagation operators and $\lim_{n\rightarrow\infty}d(X_n,\bigsqcup_{i=1}^{n-1}X_i)=\infty$, there exists $n_0'$ such that $(I-P_n)TP_n=P_nT(I-P_n)=0$ for all $T\in M$ and $n\geq n_0'$.
Thus $||TP_n-P_nT||<\varepsilon$ and $||P_n\xi-\xi||<\varepsilon$ for all $T\in M$, $\xi\in F$, and $n\geq\max(n_0,n_0')$.
\end{proof}

\subsection{Proper infiniteness}

We will show that proper infiniteness in standard form of the $\ell^p$ uniform Roe algebra $B^p_u(X)$ of a bounded geometry metric space $X$ is equivalent to $X$ being non-amenable. We also prove other characterizations involving $L^p$ Cuntz algebras, the $K_0$ group of $B^p_u(X)$, and normalized traces on $B^p_u(X)$.

\begin{defn}
Let $(X,d)$ be a metric space with bounded geometry. For $A\subset X$ and $R>0$, let $\partial_R(A)=\{ x\in X: d(x,A)\leq R\;\text{and}\;d(x,X\setminus A)\leq R \}$.
\begin{enumerate}
\item Let $R>0$ and $\varepsilon>0$. A nonempty finite set $F\subset X$ is called an $(R,\varepsilon)$-F\o lner set if $\frac{|\partial_R F|}{|F|}\leq\varepsilon$. 
\item $(X,d)$ is said to be amenable if there exists an $(R,\varepsilon)$-F\o lner set for every $R>0$ and $\varepsilon>0$.
\end{enumerate}
\end{defn}

Note that amenability is a coarse invariant for (pseudo-)metric spaces with bounded geometry \cite[Proposition 3.D.33]{CornHarpe}.

\begin{defn}
Let $(X,d)$ be a metric space. A partial translation on $X$ is a triple $(A,B,t)$, where $A,B\subset X$ and $t:A\rightarrow B$ is a bijection such that the graph of $t$ is controlled, i.e., $\sup_{x\in A}d(x,t(x))<\infty$. 

We call $A$ the domain of $t$, denoted by $\mathrm{dom}(t)$, and we call $B$ the range of $t$, denoted by $\mathrm{ran}(t)$.

Given two partial translations $t$ and $t'$, we may form their composition $t\circ t'$ by restricting the domain to $(t')^{-1}(\mathrm{dom}(t)\cap \mathrm{ran}(t'))$ and restricting the range to $t(\mathrm{dom}(t)\cap\mathrm{ran}(t'))$.

\end{defn}

Note that any partial translation $t$ on $X$ gives rise to an operator $V_t$ on $\ell^p(X)$ with finite propagation given by
\[ (V_t)_{xy}=\begin{cases} 1 &\text{if $x=t(y)$} \\ 0 &\text{otherwise} \end{cases}. \]
In fact, the linear span of such operators is dense in $B^p_u(X)$ (cf. \cite[Section 4]{amenUR}).

\begin{defn}
A mean $\mu:\mathcal{P}(X)\rightarrow[0,1]$ on a metric space $(X,d)$ is a normalized, finitely additive map on the set of all subsets of $X$. A mean is said to be invariant under partial translations if $\mu(A)=\mu(B)$ for all partial translations $(A,B,t)$ on $(X,d)$.
\end{defn}

By the Riesz representation theorem, any mean $\mu$ on $(X,d)$ induces a linear functional $\phi_\mu:\ell^\infty(X)\rightarrow\mathbb{C}$ of norm one such that $\mu(Y)=\phi_\mu(1_Y)$ for any subset $Y\subset X$, where $1_Y$ denotes the characteristic function of $Y$. Moreover, $\mu$ is invariant under partial translations if and only if $\phi_\mu(f)=\phi_\mu(f\circ t)$ for any partial translation $t$ and any $f\in\ell^\infty(X)$ supported on $\mathrm{ran}(t)$.

\begin{defn}
Let $(X,d)$ be a metric space. A paradoxical decomposition of $X$ is a partition $X=X_+\sqcup X_-$ such that there exist two partial translations $t_\pm:X\rightarrow X_\pm$.
\end{defn}

\begin{thm}\cite[Theorem 2.17]{AmenCoa}
Let $(X,d)$ be a metric space with bounded geometry. Then the following are equivalent:
\begin{enumerate}
\item $(X,d)$ is amenable.
\item $X$ admits no paradoxical decomposition.
\item There exists a mean $\mu$ on $X$ that is invariant under partial translations.
\end{enumerate}
\end{thm}

\begin{defn}
Let $A$ be a unital Banach algebra. Then a normalized trace on $A$ is a linear functional $\tau$ on $A$ satisfying the following conditions:
\begin{enumerate}
\item $\tau(1_A)=1$,
\item $||\tau||=1$,
\item $\tau(ba)=\tau(ab)$ for all $a,b\in A$.
\end{enumerate}
\end{defn}
When $A$ is a unital $C^*$-algebra, normalized traces are exactly the tracial states.

\begin{lem}\cite[Lemma 1.6]{DLR} \label{Lem:comm}
Let $A$ be a unital, properly infinite algebra in standard form. Then $1_A$ is the sum of two commutators.
\end{lem}

\begin{proof}
Take an idempotent element $e\in A$ such that $e\sim 1_A\sim 1_A-e$. Take $a_1,b_1,a_2,b_2\in A$ such that $a_1b_1=e$, $a_2b_2=1_A-e$, and $b_1a_1=1_A=b_2a_2$. Then $[b_1,a_1]+[b_2,a_2]=1_A-e+1_A-(1_A-e)=1_A$.
\end{proof}

\begin{cor} \label{Cor:PInotrace}
Let $A$ be a unital algebra such that $M_n(A)$ is properly infinite in standard form for some $n\in\mathbb{N}$. Then there is no linear functional $\tau$ on $A$ satisfying $\tau(1_A)=1$ and $\tau(ab)=\tau(ba)$ for all $a,b\in A$. 

In particular, if $A$ is a unital Banach algebra such that $M_n(A)$ is properly infinite in standard form for some $n\in\mathbb{N}$, then there are no normalized traces on $A$.
\end{cor}

\begin{proof}
Suppose $\tau$ is a linear functional on $A$ satisfying $\tau(1_A)=1$ and $\tau(ab)=\tau(ba)$ for all $a,b\in A$. Then for each $n\in\mathbb{N}$, $\tau$ extends to a linear functional $\tau_n$ on $M_n(A)$ satisfying $\tau_n(I_n)=n\tau(1_A)=n$, and $\tau_n(xy)=\tau_n(yx)$ for all $x,y\in M_n(A)$. On the other hand, if $M_n(A)$ is properly infinite in standard form, then the identity $I_n$ is the sum of two commutators by Lemma \ref{Lem:comm}, so $\tau_n(I_n)=0$.
\end{proof}

\begin{rem} 
The situation is different if we consider bounded (not necessarily normalized) traces, i.e., bounded linear functionals $\tau$ satisfying $\tau(ab)=\tau(ba)$. There are no nonzero (not necessarily normalized) traces on a unital, properly infinite $C^\ast$-algebra because every element in such an algebra is a sum of two commutators \cite[Remark 3]{Pop}. However, there can be nonzero bounded traces on a unital, properly infinite Banach algebra in general. In \cite{DLR}, there are examples of unital Banach $\ast$-algebras that are properly infinite in standard form and have nonzero, bounded traces. Moreover, they can be hermitian or $\ast$-semisimple (see \cite[Definition 1.10]{DLR}). Note that all unital $C^\ast$-algebras are hermitian \cite[Proposition 3.2.3(v)]{Dales} and $\ast$-semisimple \cite[Corollary 3.2.13]{Dales}.
\end{rem}

The final ingredient of our theorem is the $L^p$ analog of Cuntz algebras that were defined by Phillips in \cite{Phil12}. These are norm closures of certain representations of Leavitt algebras on $L^p$ spaces.

For $d\in\{2,3,4,\ldots\}$, define the Leavitt algebra $L_d$ to be the universal complex associative algebra on generators $s_1,\ldots,s_d,t_1,\ldots,t_d$ satisfying the relations
\begin{enumerate}
\item $t_js_j=1$ for $j\in\{1,\ldots,d\}$,
\item $t_js_k=0$ if $j\neq k$,
\item $\sum_{j=1}^d s_jt_j=1$.
\end{enumerate}
This is a special case of the algebras introduced by Leavitt \cite{Leavitt62,Leavitt65} who considered these algebras over arbitrary fields.

\begin{defn}
Let $A$ be a unital complex algebra, and let $E$ be a nonzero complex Banach space. We say that $\pi$ is a representation of $A$ on $E$ if $\pi$ is a unital algebra homomorphism from $A$ to $B(E)$.
\end{defn}

For $p\in[1,\infty]$, we say that a Banach algebra $A$ is an $L^p$ operator algebra if there is a measure space $(X,\mu)$ and an isometric representation of $A$ on $L^p(X,\mu)$. Clearly, $B^p_u(X)$ is an $L^p$ operator algebra on $l^p(X)$. We now recall the definition of $L^p$ Cuntz algebras, which are also $L^p$ operator algebras.

\begin{defn} \cite[Definition 6.3 and Definition 6.4]{Phil12}
Let $(X,\mathcal{B},\mu)$ and $(Y,\mathcal{C},\nu)$ be $\sigma$-finite measure spaces, and let $p\in[1,\infty]$. A linear map $s\in B(L^p(X,\mu),L^p(Y,\nu))$ is called a spatial partial isometry if there exists a quadruple $(E,F,S,g)$ in which $E\in\mathcal{B}$, $F\in\mathcal{C}$, $S$ is a bijective measurable set transformation from $(E,\mathcal{B}|_E,\mu|_E)$ to $(F,\mathcal{C}|_F,\nu|_F)$ such that $\nu|_F$ is $\sigma$-finite, $g$ is a $\mathcal{C}$-measurable function on $F$ such that $|g(y)|=1$ for almost all $y\in F$, and
\[ (s\xi)(y)=\begin{cases} g(y)\left(\left[ \frac{dS_*(\mu|_E)}{d(\nu|_F)} \right](y)\right)^{1/p}S_*(\xi|_E)(y) &\text{if}\;y\in F \\ 0 &\text{if}\;y\notin F \end{cases}. \]
\end{defn}

Given a spatial partial isometry $s\in B(L^p(X,\mu),L^p(Y,\nu))$ as defined above, there exists a unique spatial partial isometry $t\in B(L^p(Y,\nu),L^p(X,\mu))$ with quadruple $(F,E,S^{-1},(S^{-1})_*({g})^{-1})$. Moreover, the operator $ts$ is multiplication by the characteristic function of $E$, while the operator $st$ is multiplication by the characteristic function of $F$ (see \cite[Lemma 6.12]{Phil12}). The spatial partial isometry $t$ is called the reverse of $s$.

\begin{defn} \cite[Definition 7.4]{Phil12}
Let $p\in[1,\infty]$, let $(X,\mu)$ be a $\sigma$-finite measure space, and let $\rho:L_d\rightarrow B(L^p(X,\mu))$ be a representation.
We say that $\rho$ is spatial if for each $j$, the operators $\rho(s_j)$ and $\rho(t_j)$ are spatial partial isometries, with $\rho(t_j)$ being the reverse of $\rho(s_j)$. 
\end{defn}

In \cite[Theorem 8.7]{Phil12}, it is shown that any two spatial representations of $L_d$ on $L^p(X,\mu)$ with $(X,\mu)$ $\sigma$-finite give isometrically isomorphic Banach algebras, so the following definition makes sense.

\begin{defn} \cite[Definition 8.8]{Phil12}\label{Op_d}
Let $d\in\{2,3,4,\ldots\}$ and let $p\in[1,\infty)$. Define $\mathcal{O}^p_d$ to be the completion of $L_d$ in the norm $a\mapsto||\rho(a)||$ for any spatial representation $\rho$ of $L_d$ on a space of the form $L^p(X,\mu)$, where $(X,\mu)$ is $\sigma$-finite. The algebras $\mathcal{O}^p_d$ are called the $L^p$ Cuntz algebras.
\end{defn}

When $p=2$, we get the original Cuntz algebra $\mathcal{O}_d$ introduced in \cite{Cuntz77}.

The following theorem is an analog of \cite[Theorem 4.9]{amenUR}, which deals with the $p=2$ case.

\begin{thm}	\label{nonamenThm}
Let $(X,d)$ be a metric space with bounded geometry, and let $p\in[1,\infty)$. The following are equivalent:
\begin{enumerate}
\item $B^p_u(X)$ is properly infinite in standard form.
\item $M_n(B^p_u(X))$ is properly infinite in standard form for some $n\in\mathbb{N}$.
\item $M_n(B^p_u(X))$ is properly infinite in standard form for all $n\in\mathbb{N}$.
\item $X$ is non-amenable.
\item There is a unital isometric embedding of $\mathcal{O}^p_2$ into $B^p_u(X)$.
\item $[1]_0=[0]_0$ in the algebraic $K_0$ group $K_0(B^p_u(X))$.
\item $B^p_u(X)$ has no normalized trace.
\item There is no unital linear functional $\psi$ of norm one on $B(\ell^p(X))$ such that $\psi(aT)=\psi(Ta)$ for all $T\in B(\ell^p(X))$ and all $a\in B^p_u(X)$.
\end{enumerate}
\end{thm}
\begin{rem}
When $p=2$, by \cite[Theorem 4.9]{amenUR}, we can replace ``properly infinite in standard form'' by ``properly infinite'' in statements (i), (ii), and (iii).
\end{rem}

\begin{proof} The equivalence between (i) and (iii) is straightforward. We will show (iv) $\Rightarrow$ (v) $\Rightarrow$ (vi) $\Rightarrow$ (vii) $\Rightarrow$ (viii) $\Rightarrow$ (iv), and (v) $\Rightarrow$ (i) $\Rightarrow$ (ii) $\Rightarrow$ (vii). 

(iv) $\Rightarrow$ (v): 
The case $p=2$ follows from \cite[Theorem 4.9]{amenUR} so we will consider $p\in[1,\infty)\setminus\{2\}$.
If $X$ is non-amenable, then it has a paradoxical decomposition $X=X_1\sqcup X_2$ with partial translations $P_i:X\rightarrow X_i$. These partial translations give rise to elements $S_i\in B^p_u(X)$. Similarly, $P_i^{-1}:X_i\rightarrow X$ are partial translations, and they give rise to elements $T_i\in B^p_u(X)$. The elements $T_i$ and $S_i$ satisfy the following relations:
\begin{enumerate}[(I)]
\item $T_iS_i=1$ for $i=1,2$,
\item $T_iS_j=0$ if $i\neq j$,
\item $S_1T_1+S_2T_2=1$.
\end{enumerate}
Thus there is a unital homomorphism $\rho$ from the Leavitt algebra $L_2$ to $B^p_u(X)$ sending the generators $s_i$ and $t_i$ to $S_i$ and $T_i$ respectively. Using \cite[Lemma 6.16]{Phil12} (or directly from the definition), one can check that each $S_i$ is a spatial partial isometry. Then by \cite[Theorem 7.7]{Phil12}, $\rho$ is a spatial representation. Hence it extends to a unital isometric homomorphism from $\mathcal{O}^p_2$ to $B^p_u(X)$.

(v) $\Rightarrow$ (vi): Let $\rho:L_d\rightarrow L^p(X,\mu)$ be any spatial representation with $(X,\mu)$ $\sigma$-finite. Then we have $\rho(t_1s_1)=\rho(t_2s_2)=1$ and $\rho(s_1t_1+s_2t_2)=1$ in $\mathcal{O}^p_2$, and thus in $B^p_u(X)$. Moreover, $\rho(s_1t_1)$ and $\rho(s_2t_2)$ are orthogonal idempotents in $\mathcal{O}^p_2$, and thus in $B^p_u(X)$. Now 
\begin{align*}
[1]_0 &= [\rho(s_1t_1)+\rho(s_2t_2)]_0 \\
&= [\rho(s_1t_1)]_0+[\rho(s_2t_2)]_0 \\
&= [\rho(t_1s_1)]_0+[\rho(t_2s_2)]_0 \\
&= [1]_0+[1]_0.
\end{align*}
Hence $[1]_0=[0]_0$ in $K_0(B^p_u(X))$.

(vi) $\Rightarrow$ (vii): If $\tau:B^p_u(X)\rightarrow\mathbb{C}$ is a normalized trace, then it induces a group homomorphism $\tau_*:K_0(B^p_u(X))\rightarrow\mathbb{R}$ with $\tau_*([1]_0)=\tau(1)=1$ while $\tau_*([0]_0)=\tau(0)=0$.

(vii) $\Rightarrow$ (viii): If $\psi$ satisfies the conditions in (viii), then its restriction to $B^p_u(X)$ is a normalized trace.

(viii) $\Rightarrow$ (iv): 
Note that $B^p_u(X)$ contains $\ell^\infty(X)$ as the subalgebra of diagonal matrices, where elements of $\ell^\infty(X)$ act as multiplication operators on $\ell^p(X)$, and there is a conditional expectation \[E_0:B(\ell^p(X))\rightarrow\ell^\infty(X),\] i.e., $E_0$ is a linear map satisfying
\begin{itemize}
\item $||E_0||=1$,
\item $E_0(a)=a$ for all $a\in \ell^\infty(X)$,
\item $E_0(aTc)=aE_0(T)c$ for all $T\in B(\ell^p(X))$ and $a,c\in \ell^\infty(X)$.
\end{itemize}
In fact, $E_0$ is given by the formula $E_0(T)(x)=(T\delta_x)(x)$ for $T\in B(\ell^p(X))$ and $x\in X$.

Suppose $X$ is amenable, and let $\phi_\mu:\ell^\infty(X)\rightarrow\mathbb{C}$ be the linear functional of norm one associated to a mean $\mu$ on $(X,d)$ that is invariant under partial translations. Consider 
\[\psi=\phi_\mu\circ E_0:B(\ell^p(X))\rightarrow\mathbb{C}.\]
Then $\psi(1)=1$ and $||\psi||=1$. To show that $\psi(aT)=\psi(Ta)$ for all $a\in B^p_u(X)$ and $T\in B(\ell^p(X))$, it suffices to consider the case where $a$ is of the form
\[ (V_t)_{xy}=\begin{cases} 1 &\text{if $x=t(y)$} \\ 0 &\text{otherwise} \end{cases}, \]
where $t$ is a partial translation on $X$. This is because the linear span of such operators is dense in $B^p_u(X)$ (cf. \cite[Section 4]{amenUR}).

A straightforward computation shows that for $V_t$ as above and $T=(T_{xy})_{x,y\in X}\in B(\ell^p(X))$,
\begin{align*}
E_0(TV_t)(x) &= \begin{cases} T_{t^{-1}(x),x} &\text{if $x\in\mathrm{ran}(t)$} \\ 0 &\text{otherwise} \end{cases} \\
E_0(V_tT)(x) &= \begin{cases} T_{x,t(x)} &\text{if $x\in\mathrm{dom}(t)$} \\ 0 &\text{otherwise} \end{cases}
\end{align*}
Since $E_0(V_tT)=E_0(TV_t)\circ t$, we have \[\psi(TV_t)=\phi_\mu(E_0(TV_t))=\phi_\mu(E_0(V_tT))=\psi(V_tT).\]

(v) $\Rightarrow$ (i): This follows from $\mathcal{O}^p_2$ being properly infinite in standard form.

(i) $\Rightarrow$ (ii) is trivial.

(ii) $\Rightarrow$ (vii): Apply Corollary \ref{Cor:PInotrace}.

\end{proof}

\begin{rem}
In \cite{amenUR}, the authors also considered F\o lner-type conditions in the context of Hilbert space operators, and showed that the statements in the theorem above are all equivalent to the uniform Roe algebra being a F\o lner $C^\ast$-algebra when $p=2$. However, the definitions involve the Hilbert-Schmidt norm (or equivalently, any of the Schatten $p$-norms) and it is not clear to us what the analogous definition should be in the context of operators on $L^p$ space when $p\neq 2$.
\end{rem}

We end this section by establishing a bijective correspondence between means on $X$ that are invariant under partial translations and normalized traces on $B^p_u(X)$.

The proof of the following lemma is exactly the same as that of the $p=2$ case so we omit it and refer the reader to \cite[Lemma 4.16]{amenUR}.

\begin{lem}
Let $(X,d)$ be a metric space with bounded geometry, and let $p\in[1,\infty)$.
Any normalized trace $\tau$ on $B^p_u(X)$ is given by $\tau=\tau|_{\ell^\infty(X)}\circ E$, where $E:B^p_u(X)\rightarrow\ell^\infty(X)$ is given by $E(T)(x)=(T\delta_x)(x)$ for $T\in B^p_u(X)$ and $x\in X$. 
\end{lem}

\begin{prop}\label{unique trace}
Let $(X,d)$ be a metric space with bounded geometry, and let $p\in[1,\infty)$.
\begin{enumerate}
\item Any normalized trace $\tau$ on $B^p_u(X)$ extends to some unital linear functional $\psi$ of norm one on $B(\ell^p(X))$ such that $\psi(aT)=\psi(Ta)$ for all $T\in B(\ell^p(X))$ and all $a\in B^p_u(X)$.
\item There is a bijective correspondence between means on $X$ that are invariant under partial translations and normalized traces on $B^p_u(X)$.
\end{enumerate}
\end{prop}

\begin{proof}
For (i), we have $\tau=\tau|_{\ell^\infty(X)}\circ E$ by the lemma. 
Let $E_0:B(\ell^p(X))\rightarrow\ell^\infty(X)$ be the conditional expectation, and define $\psi=\tau|_{\ell^\infty(X)}\circ E_0$. Following the proof of (viii) $\Rightarrow$ (iv) in Theorem \ref{nonamenThm} with $\tau|_{\ell^\infty(X)}$ in place of $\phi_\mu$, one sees that $\psi$ has the desired properties.

For (ii), given a mean $\mu$ on $X$ that is invariant under partial translations, let $\phi_\mu:\ell^\infty(X)\rightarrow\mathbb{C}$ be the linear functional of norm one associated to $\mu$. Then $\tau_\mu:=\phi_\mu\circ E$ is a normalized trace on $B^p_u(X)$.
On the other hand, given a normalized trace $\tau$ on $B^p_u(X)$, the formula $\mu_\tau(Y)=\tau|_{\ell^\infty(X)}(1_Y)$ for $Y\subseteq X$ defines a mean on $X$, and it is invariant under partial translations by the trace property. 
For all $Y\subseteq X$, 
\[\mu_{\tau_\mu}(Y)=(\tau_\mu)|_{\ell^\infty(X)}(1_Y)=\phi_\mu(1_Y)=\mu(Y),\]
and by the lemma above, 
\[\tau_{\mu_\tau}=\phi_{\mu_\tau}\circ E=\tau|_{\ell^\infty(X)}\circ E=\tau.\]
Hence the maps $\mu\mapsto\tau_\mu$ and $\tau\mapsto\mu_\tau$ are inverses of each other.
\end{proof}

\begin{cor}
Let $(X,d)$ be a metric space with bounded geometry, and let $p\in [1,\infty)$. Then the following statements are equivalent: 
\begin{itemize}
\item[(i)] $X$ has a strictly
positive mean (i.e., every non-empty open subset has strictly positive measure) which is invariant under partial translations.
\item[(ii)] $|X|<\infty$.
\item[(iii)] $B^p_u(X)$ admits a unique normalized trace.
\end{itemize}
\end{cor}
\begin{proof}
(i) $\Rightarrow$ (ii): Let $\mu$ be an invariant mean on $X$ and assume that $|X|=\infty$. For every non-empty finite subset $F$ of $X$, we fix $x_0$ in $F$. It is clear that there is a partial translation from $\{x\}$ to $\{x_0\}$ for each $x\in F$. Hence,
\begin{align*}
1\geq \mu(F)=\sum_{x\in F}\mu(\{x\})=\sum_{x\in F}\mu(\{x_0\})=|F|\mu(\{x_0\}).
\end{align*}
Since $F$ can be arbitrary large, it forces $\mu(\{x_0\})=0$. Hence, $\mu$ is not strictly positive.\\
(ii) $\Rightarrow$ (i): If $|X|=n$, we just consider $\mu(A)=\frac{1}{n}|A|$ for $A\subseteq X$.\\
(ii) $\Rightarrow$ (iii): If $|X|=n$, then $B^p_u(X)=M_n(\mathbb{C})$. Let $\tau$ be any normalized trace on $M_n(\mathbb{C})$ and $e_{ij}$ be the matrix unit for $i,j\in \{1,\ldots,n\}$. It is easy to see that $\tau(e_{ii})=\tau(e_{11})=\frac{1}{n}$ and $\tau(e_{ij})=0$ for $i\neq j$. Hence, $\tau$ is the standard normalized trace on $M_n(\mathbb{C})$.\\
(iii) $\Rightarrow$ (ii): From Theorem~\ref{nonamenThm} and Proposition~\ref{unique trace} we know that $X$ is amenable and admits a unique mean that is invariant under partial translations. Suppose that $|X|=\infty$. Then we can inductively choose a F\o lner sequence $(F_n)_{n\in \mathbb{N}}$ in $X$ consisting of disjoint subsets. Indeed, if $F_1,\ldots, F_n$ have been chosen, we consider the space $Y=X\backslash (\bigcup_{i=1}^n F_i)$. Since $\bigcup_{i=1}^n F_i$ is finite, $Y$ is an infinite amenable subspace which is coarsely equivalent to $X$. Hence there is $F_{n+1}\subset Y$ with the right property.

Now we take any free ultrafilter $\omega$ on $\mathbb{N}$ and the formula
\begin{align*}
\tau_\omega(f)=\lim_{\omega}\frac{1}{|F_n|}\sum_{x\in F_n}f(x)
\end{align*}
defines an invariant mean on $\ell^\infty(X)$ by the F\o lner property. Let $\omega_1$ and $\omega_2$ be two different ultrafilters on $\mathbb{N}$. Then there is a subset $A$ of $\mathbb{N}$ such that $A\in \omega_1$ and $A \notin \omega_2$. It follows that $\tau_{\omega_1}(f)=1$ and $\tau_{\omega_2}(f)=0$ for $f=\sum_{n\in A}1_{F_n}$ in $\ell^\infty(X)$. This is a contradiction to the uniqueness.
\end{proof}

\section{$K$-theory of properly infinite $\ell^p$ uniform Roe algebras}
In this section, we study the homomorphism $K_0(\ell^\infty(X))\rightarrow K_0(B^p_u(X))$ induced by the canonical diagonal inclusion $\ell^\infty(X)\hookrightarrow B^p_u(X)$. When the underlying metric space $X$ has asymptotic dimension at most one, we show that this homomorphism is always surjective. As a consequence, if $X$ is a non-amenable metric space with asymptotic dimension at most one then $K_0(B^p_u(X))=0$ for all $p\in[1,\infty)$.

\begin{lem}\cite[Lemma~17]{MR2363428}
Let $(X,d)$ be a metric space with bounded geometry. Then for any $R>0$, the space $X$ can be written as a finite disjoint union of $R$-separated subsets, i.e., $X=\coprod_{i=1}^nX_i$ and for each $i$ we have $d(x,y)\geq R$ whenever $x,y$ are distinct points in $X_i$.
\end{lem}


Using the lemma above, given $R>0$, we can find a finite set $\mathcal{T}$ of partial translations such that for any pair of points $x,y\in X$ with $d(x,y)\leq R$, there exists $t\in\mathcal{T}$ such that $t(y)=x$. These partial translations can be described in the following way:

Fix $R>0$, and write $X$ as a disjoint union $X=\coprod_{i=1}^nX_i$ of $S$-separated sets with $S>2R$. Let $t_{ij}$ be the map that sends $y\in X_j$ to $x\in X_i$ if $d(x,y)\leq R$. For each $x\in X_i$, there is at most one $y\in X_j$ with $d(x,y)\leq R$ since $X_j$ is $S$-separated with $S>2R$. Conversely, for each $y\in X_j$ there is at most one $x\in X_i$ with $d(x,y)\leq R$. Hence $t_{ij}$ is a partial translation. Moreover, observe that if $T_{ij}\in B^p_u(X)$ is the partial isometry associated to $t_{ij}$, then $T_{ji}$ is the transpose of $T_{ij}$.

We will use these partial translations in the proof of the following lemma, which generalizes Lemma 2.5 and Lemma 2.6 in \cite{MatuiRordam} dealing with the case where $X$ is a discrete group and $p=2$.

\begin{lem}
Let $(X,d)$ be a metric space with bounded geometry, let $A$ and $B$ be nonempty subspaces of $X$, and let $p\in[1,\infty)$. Then $1_A$ is in the closed two-sided ideal in $B^p_u(X)$ generated by $1_B$ if and only if there exists $R<\infty$ such that $d(x,B)\leq R$ for all $x\in A$.
\end{lem}

\begin{proof}
($\Rightarrow$): Suppose that $1_A$ belongs to the closed two-sided ideal generated by $1_B$. Then there exist finite propagation operators $T_1,\ldots,T_n,S_1,\ldots,S_n\in B^p_u(X)$ such that $||1_A-\sum_{j=1}^n T_j1_B S_j||<1$. In particular, \[||1_A-\sum_{j=1}^n E(T_j1_BS_j)||<1,\] where $E:B^p_u(X)\rightarrow\ell^\infty(X)$ is the conditional expectation. 

Now note that if $T,S\in B^p_u(X)$ have finite propagation, then $E(T1_BS)(x)=\sum_{z\in B}T_{x,z}S_{z,x}=0$ if $d(x,B)>\mathrm{prop}(T)$. Thus for all $x\in A$, we have
\[d(x,B)\leq\max_{1\leq j\leq n}\{\mathrm{prop}(T_j)\}.\]

($\Leftarrow$): Suppose that there exists $R<\infty$ such that $d(x,B)\leq R$ for all $x\in A$. Using this $R$ and writing $X$ as a disjoint union $X=\coprod_{i=1}^nX_i$ of $S$-separated sets with $S>2R$, we get a finite set of partial translations $t_{ij}$ as described above. Let us consider the associated spatial partial isometries $T_{ij}$ in $B^p_u(X)$. We claim that $1_A=f1_A\sum_{i,j=1}^n T_{ji}1_BT_{ij}$ for some $f\in\ell^\infty(X)\subset B^p_u(X)$. It then follows that $1_A$ is in the closed two-sided ideal in $B^p_u(X)$ generated by $1_B$.

To prove the claim, first observe that for $x,y\in X$, we have \[ (T_{ji}1_BT_{ij})_{xy}=\sum_{z\in B}(T_{ji})_{xz}(T_{ij})_{zy}=\begin{cases} \sum_{z\in B}(T_{ij})_{zy} &\text{if}\; x=y \\ 0 &\text{if}\; x\neq y \end{cases}, \] where we have used the fact that $T_{ji}$ is the transpose of $T_{ij}$, and that all entries in $T_{ij}$ are either 0 or 1 with at most one nonzero entry in each row. In particular, $T_{ji}1_BT_{ij}\in \ell^\infty(X)\subset B^p_u(X)$ and \[ (1_AT_{ji}1_BT_{ij})_{yy}=\begin{cases} \sum_{z\in B}(T_{ij})_{zy} &\text{if}\; y\in A\cap X_j \\ 0 &\text{if}\; y\notin A\cap X_j \end{cases}. \]
Note that $\sum_{z\in B}(T_{ij})_{zy}$ is either 0 or 1 for each $y\in A$. Also, for each $y\in A\cap X_j$ and $\varepsilon>0$, there exists $z\in B$ such that $d(y,z)< R+\varepsilon$, and $z\in X_i$ for some $i$, so we have $(\sum_{i,j=1}^n 1_AT_{ji}1_BT_{ij})_{yy}\in\{1,2,\ldots,n\}$ for each $y\in A$. Hence $1_A=f1_A\sum_{i,j=1}^n T_{ji}1_BT_{ij}$ for some $f\in\ell^\infty(X)\subset B^p_u(X)$.
\end{proof}

Recall that an idempotent in a Banach algebra $A$ is \emph{full} if it is not contained in any proper, closed, two-sided ideal in $A$.

\begin{lem} (see \cite[Lemma 5.3]{lowdimUR} for $p=2$ case) \label{Lem:diagidem1}
Let $p\in[1,\infty)$. If $(X,d)$ is a non-amenable metric space with bounded geometry, then $[e]_0=[0]_0$ in $K_0(B^p_u(X))$ for every idempotent $e\in\ell^\infty(X)$ that is full in $B^p_u(X)$.
\end{lem}
\begin{proof}
Let $e=1_A$ be an idempotent in $\ell^\infty(X)$ that is full in $B^p_u(X)$. Then $e$ and $1_X$ generate the same closed two-sided ideal in $B^p_u(X)$, so there exists $R>0$ such that $d(x,A)\leq R$ for all $x\in X$ by the previous lemma. In particular, the inclusion $A\hookrightarrow X$ is a quasi-isometry. Thus $A$ is non-amenable, so we have $[e]_0=[0]_0$ in $K_0(B^p_u(A))$ by Theorem \ref{nonamenThm}. The inclusion map $B^p_u(A)\hookrightarrow B^p_u(X)$ maps $e$ to the diagonal idempotent $1_A$ in $B^p_u(X)$ so $[e]_0=[0]_0$ in $K_0(B^p_u(X))$.
\end{proof}

\begin{lem} (see \cite[Lemma 5.4]{lowdimUR} for $p=2$ case)  \label{Lem:diagidem2}
If $(X,d)$ is a non-amenable metric space with bounded geometry, then $[e]_0=[0]_0$ in $K_0(B^p_u(X))$ for every idempotent $e\in\ell^\infty(X)$ and every $p\in[1,\infty)$. 

Hence, the homomorphism $K_0(\ell^\infty(X))\rightarrow K_0(B^p_u(X))$ induced by the diagonal inclusion $\ell^\infty(X)\hookrightarrow B^p_u(X)$ is the zero map for all $p\in[1,\infty)$.
\end{lem}

\begin{proof}
By assumption, $X$ admits a paradoxical decomposition $X=X_1\sqcup X_2$ with partial translations $t_i:X\rightarrow X_i$ for $i=1,2$. These partial translations give rise to $S_i\in B^p_u(X)$, and their inverses give rise to $T_i\in B^p_u(X)$ such that $T_iS_i=1_X$ and $S_iT_i=1_{X_i}$ for $i=1,2$. Since $1_X=T_i 1_{X_i} S_i$, we see that $1_{X_1}$ and $1_{X_2}$ are full idempotents in $B^p_u(X)$.

Let $e=1_A$ be an idempotent in $\ell^\infty(X)$, where $A$ is a subspace of $X$. Then $e=1_{A\cap X_1}+1_{A\cap X_2}$.
Since $(1_X-1_{A\cap X_2})1_{X_1}=1_{X_1}=1_{X_1}(1_X-1_{A\cap X_2})$ and $(1_X-1_{A\cap X_1})1_{X_2}=1_{X_2}=1_{X_2}(1_X-1_{A\cap X_1})$, we conclude that $1_X-1_{A\cap X_1}$ and $1_X-1_{A\cap X_2}$ are full idempotents in $B^p_u(X)$. By Lemma \ref{Lem:diagidem1}, $[1_X-1_{A\cap X_1}]_0=[1_X-1_{A\cap X_2}]_0=[0]_0$ in $K_0(B^p_u(X))$. Since $X$ is non-amenable, we also have $[1_X]_0=[0]_0$ in $K_0(B^p_u(X))$ by Theorem \ref{nonamenThm}. Hence $[e]_0=[1_{A\cap X_1}]_0+[1_{A\cap X_2}]_0=[0]_0$ in $K_0(B^p_u(X))$.

To show that the homomorphism $K_0(\ell^\infty(X))\rightarrow K_0(B^p_u(X))$ induced by the diagonal inclusion is zero, note that there is a group isomorphism
\[ \dim:K_0(C(\beta X))\rightarrow C(\beta X,\mathbb{Z}) \]
satisfying $\dim([e]_0)(x)=Tr(e(x))$ for every $x\in\beta X$, $e$ an idempotent in $M_n(C(\beta X))$ for some $n\in\mathbb{N}$, and $Tr$ denoting the standard non-normalized trace on $M_n(\mathbb{C})$. In particular, every idempotent in $M_n(\ell^\infty(X))$ is equivalent to a direct sum of idempotents in $\ell^\infty(X)$. Since every idempotent in $\ell^\infty(X)$ has the trivial $K_0$ class in $K_0(B^p_u(X))$, so does every idempotent in $M_n(\ell^\infty(X))$.
\end{proof}

The main result of this section is the following theorem:
\begin{thm} (see \cite[Theorem 5.2]{lowdimUR} for $p=2$ case) \label{thm:lowdim}
Let $(X,d)$ be a metric space with bounded geometry. If the asymptotic dimension of $X$ is at most one, then the homomorphism $K_0(\ell^\infty(X))\rightarrow K_0(B^p_u(X))$ induced by the canonical diagonal inclusion is always surjective for all $p\in[1,\infty)$.
\end{thm}

Before proving the theorem, we note the following result, which is an immediate consequence of Lemma \ref{Lem:diagidem2} and Theorem \ref{thm:lowdim}.

\begin{cor} (see \cite[Corollary 5.5]{lowdimUR} for $p=2$ case)\label{non-amen ass one}
If $(X,d)$ is a non-amenable metric space with bounded geometry and it has asymptotic dimension (at most) one, then $K_0(B^p_u(X))=0$ for all $p\in[1,\infty)$.
\end{cor}

The proof of Theorem \ref{thm:lowdim} uses a controlled or quantitative version of $K$-theory for $L^p$ operator algebras from \cite{Chung1}. We state some of the necessary results that we shall use, and refer the reader to \cite{Chung1} for more details.

\begin{defn}
An $L^p$ operator algebra $A$ is filtered if it has a family $(A_r)_{r\geq 0}$ of linear subspaces indexed by non-negative real numbers $r\in [0,\infty)$ such that
\begin{itemize}
\item $A_r\subseteq A_{r'}$ if $r'\geq r$;
\item $A_r A_{r'}\subseteq A_{r+r'}$ for all $r,r'\geq 0$;
\item $\bigcup_{r\geq 0}A_r$ is dense in $A$.
\end{itemize}
If $A$ is unital with unit $1_A$, we require $1_A\in A_r$ for all $r\geq 0$. The family $(A_r)_{r\geq 0}$ is called a filtration of $A$.
\end{defn}

Given a filtered $L^p$ operator algebra $A$ with a filtration $(A_r)_{r\geq 0}$, one may consider $(\varepsilon,r,N)$-idempotents and $(\varepsilon,r,N)$-invertibles, where $0<\varepsilon<\frac{1}{20}$, $r\geq 0$, and $N\geq 1$. Then one defines appropriate homotopy relations and goes through the standard procedure in the definition of $K$-theory of Banach algebras to arrive at the controlled $K$-theory groups $K_0^{\varepsilon,r,N}(A)$ and $K_1^{\varepsilon,r,N}(A)$. For each triple $(\varepsilon,r,N)$, we have homomorphisms $c:K_*^{\varepsilon,r,N}(A)\rightarrow K_*(A)$. In the case of $K_0$, this map is given by applying holomorphic functional calculus to an $(\varepsilon,r,N)$-idempotent to obtain an idempotent; in the case of $K_1$, this map arises from the fact that every $(\varepsilon,r,N)$-invertible is invertible. 

The next two lemmas describe properties of these homomorphisms that we will need in the proof of Theorem \ref{thm:lowdim}.
\begin{lem} \cite[Proposition 3.20]{Chung1} \label{qKtoKsurj}
Let $A$ be a unital filtered $L^p$ operator algebra. Let $u$ be an invertible element in $M_n(A)$, and let $0<\varepsilon<\frac{1}{20}$. Then there exist $r>0$, $N\geq 1$ (depending only on $||u||$ and $||u^{-1}||$), and $[v]\in K_1^{\varepsilon,r,N}(A)$ with $v$ an $(\varepsilon,r,N)$-invertible in $M_n(A)$ such that $[v]=[u]$ in $K_1(A)$.
\end{lem}

\begin{lem} \cite[Proposition 3.21]{Chung1} \label{qKtoKinj}
There exists a (quadratic) polynomial $\rho$ with positive coefficients such that for any filtered $L^p$ operator algebra $A$, if $0<\varepsilon<\frac{1}{20\rho(N)}$, and $[e]_{\varepsilon,r,N},[f]_{\varepsilon,r,N}\in K_0^{\varepsilon,r,N}(A)$ satisfy $c([e])=c([f])$ in $K_0(A)$, then there exist $r'\geq r$ and $N'\geq N$ such that $[e]_{\rho(N)\varepsilon,r',N'}=[f]_{\rho(N)\varepsilon,r',N'}$ in $K_0^{\rho(N)\varepsilon,r',N'}(A)$.
\end{lem}

\begin{defn}
Let $A$ be a filtered $L^p$ operator algebra with filtration $(A_r)_{r\geq 0}$. A pair $(I,J)$ of closed ideals of $A$ is a controlled Mayer-Vietoris pair for $A$ if it satisfies the following conditions:
\begin{itemize}
\item For any $r\geq 0$, any positive integer $n$, and any $x\in M_n(A_r)$, there exist $x_1\in M_n(I\cap A_r)$ and $x_2\in M_n(J\cap A_r)$ such that $x=x_1+x_2$ and $\max(||x_1||,||x_2||)\leq ||x||$;
\item $I$ and $J$ have filtrations $(I\cap A_r)_{r\geq 0}$ and $(J\cap A_r)_{r\geq 0}$ respectively;
\item There exists $c>0$ such that for any $r\geq 0$, any $\varepsilon>0$, any positive integer $n$, any $x\in M_n(I\cap A_r)$ and $y\in M_n(J\cap A_r)$ with $||x-y||<\varepsilon$, there exists $z\in M_n(I\cap J\cap A_{cr})$ such that $\max(||z-x||,||z-y||)<c\varepsilon$.
\end{itemize}
\end{defn}

\begin{rem}
In \cite{Chung1}, there is a more general definition of a controlled Mayer-Vietoris pair that involves subalgebras instead of ideals, but we use this slightly simpler version here as it is sufficient for our purposes. We also note that this is slightly different from the definition of a uniformly excisive pair of ideals used in \cite{lowdimUR}.
\end{rem}

\begin{thm}\cite[Theorem 5.14]{Chung1} \label{MVthm}
Given a triple $(\varepsilon_0,r_0,N_0)\in(0,\frac{1}{20})\times[0,\infty)\times[1,\infty)$, there exist $(\varepsilon_1,r_1,N_1)$ and $(\varepsilon_2,r_2,N_2)$ with $\varepsilon_i\geq\varepsilon_0$, $r_i\geq r_0$, and $N_i\geq N_0$ such that for any filtered $L^p$ operator algebra $A$ and any controlled Mayer-Vietoris pair $(I,J)$ for $A$, if $x\in K_1^{\varepsilon_0,r_0,N_0}(A)$, then there exists $\partial_c x\in K_0^{\varepsilon_1,r_1,N_1}(I\cap J)$ with the property that if $\partial_c x=0$, then there exist $y\in K_1^{\varepsilon_2,r_2,N_2}(I)$ and $z\in K_1^{\varepsilon_2,r_2,N_2}(J)$ such that $x=y+z$ in $K_1^{\varepsilon_2,r_2,N_2}(A)$.
\end{thm}

Now let us return to the setting of Theorem \ref{thm:lowdim}. Let $X$ be a metric space with bounded geometry. For any subset $U$ of $X$ and $r\geq 0$, let $U^{(r)}$ denote the $r$-neighborhood of $U$, i.e., \[U^{(r)}=\{ x\in X:d(x,U)\leq r\}.\]

For an element in $M_n(B^p_u(X))$, we define its support to be the union of the supports of all its matrix entries, and its propagation to be the maximum of the propagation of its matrix entries.

Assume that $X$ has asymptotic dimension at most one. Then for any $R>0$, there is a decomposition $X=U\sqcup V$ such that each of $U=\bigsqcup_{i\in I}U_i$ and $V=\bigsqcup_{j\in J}V_j$ are disjoint unions of uniformly bounded $R$-separated sets. For such a decomposition, and for $r\geq 0$, consider the following subspaces of $B^p_u(X)$:
\begin{align*}
\mathfrak{N}(U)_r &= \{ a\in B^p_u(X): \mathrm{supp}(a)\subseteq\bigcup_{i\in I}U_i^{(r)}\times U_i^{(r)},\mathrm{prop}(a)\leq r \}, \\
\mathfrak{N}(V)_r &= \{ a\in B^p_u(X): \mathrm{supp}(a)\subseteq\bigcup_{j\in J}V_j^{(r)}\times V_j^{(r)},\mathrm{prop}(a)\leq r \}, \\
\mathfrak{N}(U\cap V)_r &= \{ a\in B^p_u(X): \mathrm{supp}(a)\subseteq\bigcup_{i\in I,j\in J}(U_i^{(r)}\cap V_j^{(r)})\times (U_i^{(r)}\cap V_j^{(r)}) \}.
\end{align*}

Now for $r\geq 0$, define
\begin{align*}
A_r &= \mathfrak{N}(U)_r+\mathfrak{N}(V)_r+\mathfrak{N}(U\cap V)_r, \\
I_r &= \mathfrak{N}(U)_r+\mathfrak{N}(U\cap V)_r, \\
J_r &= \mathfrak{N}(V)_r+\mathfrak{N}(U\cap V)_r.
\end{align*}
and let $A=\overline{\bigcup_{r\geq 0}A_r}$, $I=\overline{\bigcup_{r\geq 0}I_r}$, and $J=\overline{\bigcup_{r\geq 0}J_r}$. We observe that $A=B^p_u(X)$.

\begin{lem}
The subspaces $A_r$. $I_r$, and $J_r$ provide filtrations for $A$, $I$, and $J$ respectively. Moreover, $I$ and $J$ are ideals in $A$.
\end{lem}

\begin{proof}
Given $r,s\geq 0$, it is fairly straightforward to check that
\begin{enumerate}
\item $\mathfrak{N}(U)_r\cdot \mathfrak{N}(U)_s\subseteq \mathfrak{N}(U)_{r+s}$, \\
\item $\mathfrak{N}(V)_r\cdot \mathfrak{N}(V)_s\subseteq \mathfrak{N}(V)_{r+s}$, \\
\item $\mathfrak{N}(U\cap V)_r\cdot \mathfrak{N}(U\cap V)_s\subseteq \mathfrak{N}(U\cap V)_{r+s}$, \\
\item $\mathfrak{N}(U)_r\cdot \mathfrak{N}(V)_s\subseteq \mathfrak{N}(U\cap V)_{r+s}$, \\
\item $\mathfrak{N}(U)_r\cdot \mathfrak{N}(U\cap V)_s\subseteq \mathfrak{N}(U\cap V)_{r+s}$, \\
\item $\mathfrak{N}(V)_r\cdot \mathfrak{N}(U\cap V)_s\subseteq \mathfrak{N}(U\cap V)_{r+s}$.
\end{enumerate}
\end{proof}

\begin{lem}
The pair $(I,J)$ is a controlled Mayer-Vietoris pair for $A$.
\end{lem}

\begin{proof}
For a subset $Y$ of $X$, we will let $1_Y$ denote the characteristic function of $Y$ regarded as a diagonal element in $B^p_u(X)$. We will also write $1_Y$ for $1_Y\otimes I_n\in M_n(B^p_u(X))$.

Let $n\in\mathbb{N}$ and $r\geq 0$. Given $a\in M_n(A_r)$, we have \[a=\sum_{i\in I}1_{U_i}a+\sum_{j\in J}1_{V_j}a\] with $\sum_{i\in I}1_{U_i}a\in M_n(I_r)$, $\sum_{j\in J}1_{V_j}a\in M_n(J_r)$, and \[\max(||\sum_{i\in I}1_{U_i}a||,||\sum_{j\in J}1_{V_j}a||)\leq||a||.\]

Suppose that $a\in M_n(I_r)$ and $b\in M_n(J_r)$ are such that $||a-b||<\varepsilon$. Let $\chi$ be the characteristic function of $\bigcup_{i\in I}U_i^{(r)}$, and let $\chi'$ be the characteristic function of $\bigcup_{j\in J}V_j^{(r)}$. Then $a\chi',b\chi\in M_n(I_{2r}\cap J_{2r})$, $||a-b\chi||=||a\chi-b\chi||<\varepsilon$, and $||b-a\chi'||=||b\chi'-a\chi'||<\varepsilon$. Hence letting $c=\frac{a\chi'+b\chi}{2}\in M_n(I_{2r}\cap J_{2r})$, we have
$||a-c||<\frac{5}{2}\varepsilon$ and $||b-c||<\frac{5}{2}\varepsilon$.
\end{proof}

Now we are ready for the proof of Theorem \ref{thm:lowdim}. In the proof, we will actually use the fact that the pair of suspensions $(SI,SJ)$ is a controlled Mayer-Vietoris pair for $SA$ rather than working directly with $I$ and $J$ (cf. \cite[Remark 5.6]{Chung1}). This is because a controlled version of Bott periodicity was not proved in \cite{Chung1}.

\begin{proof}[Proof of Theorem \ref{thm:lowdim}]
The $p=2$ case was proved in \cite[Theorem 5.2]{lowdimUR} so we will assume that $p\in[1,\infty)\setminus\{2\}$.

Suppose that $x\in K_0(B^p_u(X))\cong K_1(SB^p_u(X))$, and $x$ is represented by an invertible $u\in M_n(\widetilde{SB^p_u(X)})$, where $\widetilde{SB^p_u(X)}$ denotes the unital algebra obtained by adjoining a unit to $SB^p_u(X)$. Consider the controlled Mayer-Vietoris pair $(SI,SJ)$ for $SA=SB^p_u(X)$.
By Lemma \ref{qKtoKsurj}, given $0<\varepsilon_0<\frac{1}{20}$, there exists $r_0\geq 0$, $N_0\geq 1$ (depending only on $||u||$ and $||u^{-1}||$), and a quasi-invertible $v\in M_n(\widetilde{SB^p_u(X)})$ with propagation at most $r_0$ such that $[v]\in K_1^{\varepsilon_0,r_0,N_0}(SB^p_u(X))$ and $[v]=x$ in $K_1(SB^p_u(X))$. Let $(\varepsilon_1,r_1,N_1)$ and $(\varepsilon_2,r_2,N_2)$ be associated to $(\varepsilon_0,r_0,N_0)$ as in Theorem \ref{MVthm}.

By the assumption on asymptotic dimension, there exists a decomposition $X=U\sqcup V$ such that each of $U=\bigsqcup_{i\in I}U_i$ and $V=\bigsqcup_{j\in J}V_j$ are disjoint unions of uniformly bounded $r$-separated sets with $r>3\max(r_1,r_2)$. Note that since $2r_1<r$, $\mathfrak{N}(U\cap V)_{r_1}$ is a direct product of matrix algebras of uniformly bounded sizes, and thus a Banach algebra direct limit of a directed system of finite direct sums of such matrix algebras (see \cite[Lemma 8.4]{WZ} or the proof of Theorem \ref{thmAsdim0} (i) $\Rightarrow$ (ii) in the next section).
By continuity of the $K_1$ functor,
we have \[K_0(S\mathfrak{N}(U\cap V)_{r_1})=K_1(\mathfrak{N}(U\cap V)_{r_1})=0.\] 
Consider $\partial_c[v]\in K_0^{\varepsilon_1,r_1,N_1}(SI\cap SJ)=K_0^{\varepsilon_1,r_1,N_1}(S\mathfrak{N}(U\cap V)_{r_1})$. By Lemma \ref{qKtoKinj}, $\partial_c[v]=0$ in $K_0^{\frac{1}{20},r',N'}(SI\cap SJ)$ for some $r'\geq r_1$ and $N'\geq N_1$.
Then by Theorem \ref{MVthm}, there exist $y\in K_1^{\varepsilon_2,r_2,N_2}(SI)$ and $z\in K_1^{\varepsilon_2,r_2,N_2}(SJ)$ with $[v]=y+z$ in $K_1^{\varepsilon_2,r_2,N_2}(SB^p_u(X))$, and thus $x=c(y)+c(z)$ in $K_1(SB^p_u(X))= K_0(B^p_u(X))$, where $c$ denotes the respective compositions
\begin{align*} K_1^{\varepsilon_2,r_2,N_2}(SI)&\rightarrow K_1^{\varepsilon_2,r_2,N_2}(SB^p_u(X))\rightarrow K_1(SB^p_u(X)) \\ K_1^{\varepsilon_2,r_2,N_2}(SJ)&\rightarrow K_1^{\varepsilon_2,r_2,N_2}(SB^p_u(X))\rightarrow K_1(SB^p_u(X)). \end{align*}
Recall that these compositions take a quantitative $K$-theory class represented by a quasi-invertible element (which is actually invertible) and send it to a $K$-theory class represented by the same element.

On the other hand, since $3r_2<r$, there is a factorization
\[
\begindc{\commdiag}[100]		
\obj(0,10)[1a]{$K_1^{\varepsilon_2,r_2,N_2}(SI)$}		
\obj(10,10)[1b]{$K_1^{\varepsilon_2,r_2,N_2}(SB^p_u(X))$}
\obj(20,10)[1c]{$K_1(SB^p_u(X))$} 
\obj(10,0)[2b]{$K_1(S(\prod_{i\in I}B(\ell^p(U_i^{(r_2)}))))$} 

\mor{1a}{1b}{}	\mor{1b}{1c}{} \mor{1a}{2b}{} \mor{2b}{1c}{}	
\enddc
\]
and similarly for $J$ and $V_j$. Since $\prod_{i\in I}B(\ell^p(U_i^{(r_2)})))$ is a direct product of matrix algebras of uniformly bounded sizes, any element in its $K_0$ group is equivalent to something in the image of the map on $K$-theory induced by the diagonal inclusion
\[ \ell^\infty(U^{(r_2)})\rightarrow\prod_{i\in I}B(\ell^p(U_i^{(r_2)})), \]
and similarly for $V$. Hence $x$ is in the image of the homomorphism induced by the canonical diagonal inclusion.
\end{proof}

\section{$\ell^p$ uniform Roe algebras of spaces with zero asymptotic dimension}
In this section, we study the structure and $K$-theory of $\ell^p$ uniform Roe algebras of metric spaces with zero asymptotic dimension. We start by recalling the following equivalent definition for asymptotic dimension zero (see Definition 2.1 in \cite{lowdimUR}):
\begin{defn}\label{remAsdim}
Let $(X,d)$ be a metric space with bounded geometry. For $r\geq 0$, let $\sim_r$ be the equivalence relation generated by the relation $xRy\Leftrightarrow d(x,y)\leq r$. Then $X$ has asymptotic dimension zero if and only if for each $r\geq 0$, the relation $\sim_r$ has uniformly finite equivalence classes.
\end{defn}
We will show that a metric space $X$ has asymptotic dimension zero if and only if $B^p_u(X)$ is a spatial $L^p$ AF algebra (see \cite[Definition~9.1]{PhilViola}) if and only if $B^p_u(X)$ has cancellation in the following sense: 
\begin{defn}\cite[Section 6.4]{Bl}\label{cancellation}
Let $A$ be a Banach algebra. We say that $A$ has cancellation of idempotents if whenever $e,f,g,h$ are idempotents in $A$ with $e\perp g$, $f\perp h$, $e\sim f$, and $e+g\sim f+h$, then $g\sim h$.

We say that $A$ has cancellation if $M_n(A)$ has cancellation of idempotents for all $n \in   \mathbb{N}$.
\end{defn}
We will need the following useful characterizations.
\begin{prop}\cite[Proposition 6.4.1]{Bl} \label{canc}
Let $A$ be a unital Banach algebra. Then the following are equivalent:
\begin{enumerate}
\item $A$ has cancellation of idempotents.
\item If $e,f$ are idempotents in $A$ and $e\sim f$, then $1-e\sim 1-f$.
\item If $e,f$ are idempotents in $A$ and $e\sim f$, then there exists an invertible element $u$ in $A$ such that $ueu^{-1}=f$.
\end{enumerate}
\end{prop}

The following theorem is an analog of \cite[Theorem 2.2]{lowdimUR}.

\begin{thm} \label{thmAsdim0}
Let $X$ be a metric space with bounded geometry, and let $p\in[1,\infty)$. The following are equivalent:
\begin{enumerate}
\item $X$ has asymptotic dimension zero.
\item $B^p_u(X)$ is an inductive limit of 
subalgebras isometrically isomorphic to $\bigoplus_{k=1}^N M_{d_k}^p$ with norm $||(a_1,\ldots,a_N)||=\max(||a_1||,\ldots,||a_N||)$, where $N,d_1,\ldots,d_k\in\mathbb{N}$, and $M_d^p$ denotes $B(\ell^p(\{1,\ldots,d\}))$.\footnote{This is a non-sequential analogue of a spatial $L^p$ AF algebra in the sense of \cite[Definition~9.1]{PhilViola}.}

\item $B^p_u(X)$ has cancellation.
\end{enumerate}
\end{thm}

\begin{proof}
(i) $\Rightarrow$ (ii): Assume that $X$ has asymptotic dimension zero. Then for each $r>0$, the equivalence classes for the relation $\sim_r$ in Definition \ref{remAsdim} are uniformly finite. We will write $I_r$ for the collection of all equivalence classes for $\sim_r$. Fix a total order on $X$. For each finite subset $A\subset X$, let $f_A:A\rightarrow\{1,\ldots,|A|\}$ be the order isomorphism determined by the total order.

Consider the collection $\mathcal{I}$ of ordered pairs $(r,\mathcal{P})$, where $r>0$ and $\mathcal{P}=\{P_1,\ldots,P_N\}$ is a partition of $I_r$ into finitely many nonempty sets, which we think of as colours, such that equivalence classes with the same colour have the same cardinality.
Fix $(r,\mathcal{P})$. Let $n_1,\ldots,n_N$ be the cardinalities of the sets in each of the colours $P_1,\ldots,P_N$ of $\mathcal{P}$, and let $B=\bigoplus_{i=1}^NM_{n_i}^p$.
For each $A\in I_r$ with colour $P_i$, let $u_{A,i}:\ell^p(\{1,\ldots,|A|\})\rightarrow\ell^p(A)$ be the invertible isometry determined by $f_A$. Define $\phi:B\rightarrow\prod_{A\in I_r}B(\ell^p(A))\subseteq B(\ell^p(X))$ by $(a_i)_{i=1}^N\mapsto\prod_{i=1}^N\prod_{A\in P_i}u_{A,i}a_iu_{A,i}^{-1}$.
The image of $\phi$ is contained in $B^p_u(X)$ since the sets $A\in I_r$ are uniformly bounded, and we denote this image by $A_{r,\mathcal{P}}$.

Define a partial order on $\mathcal{I}$ by $(r,\mathcal{P})\leq(s,\mathcal{Q})$ if $r\leq s$ and $A_{r,\mathcal{P}}\subseteq A_{s,\mathcal{Q}}$. We leave the reader to verify that this is indeed a partial order (or see the proof of (1) $\Rightarrow$ (2) in \cite[Theorem 2.2]{lowdimUR}).

It remains to be shown that the union $\bigcup_{(r,\mathcal{P})\in\mathcal{I}}A_{r,\mathcal{P}}$ is dense in $B^p_u(X)$.
For this, it suffices to show that any finite propagation operator in $B^p_u(X)$ can be approximated by an element of the union.
Let $\varepsilon>0$ and let $a\in B^p_u(X)$ have propagation at most $r$. Then $a$ is contained in $\prod_{A\in I_r}B(\ell^p(A))$, which we identify with $\prod_{A\in I_r}M_{|A|}^p$ using the bijections $f_A$.
Write $a_A$ for the component of $a$ in the relevant copy of $M_{|A|}^p$.
Set $N=\max\{|A|:A\in I_r\}$, and for $n\in\{1,\ldots,N\}$, choose an $(\varepsilon/2)$-dense subset $\{b_{n,1},\ldots,b_{n,m_n}\}$ of the ball of radius $||a||$ in $M_n^p$.
For each $A$, there exists $m(A)\in\{1,\ldots,m_n\}$ such that $||a_A-b_{|A|,m(A)}||<\varepsilon/2$.
Set \[P_{n,m}=\{A\in I_r:|A|=n, m(A)=m\},\] and define \[\mathcal{P}=\{P_{n,m}:n\in\{1,\ldots,N\}, m\in\{1,\ldots,M_n\}, P_{n,m}\neq\emptyset\}.\]
Then the element $b=(b_{|A|,m(A)})_{A\in I_r}$ is in $A_{r,\mathcal{P}}$ and \[||a-b||=\sup_{A\in I_r}||a_A-b_{|A|,m(A)}||<\varepsilon,\] which completes the proof.

(ii) $\Rightarrow$ (iii): This follows from the fact that $M_n^p$ has cancellation for all $n\in\mathbb{N}$, and that cancellation is preserved under taking finite direct sums and taking inductive limits.

(iii) $\Rightarrow$ (i): Assume that the asymptotic dimension of $X$ is not zero. By \cite[Lemma 2.4]{lowdimUR}, there exist $r>0$ and $S_n=\{x_1^{(n)},\ldots,x_{m_n}^{(n)}\}\subset X$ for each $n\geq 1$ with the following properties:
\begin{itemize}
\item $m_n\rightarrow\infty$ as $n\rightarrow\infty$,
\item for each $n$ and each $i\in\{1,\ldots,m_n-1\}$, $d(x_i^{(n)},x_{i+1}^{(n)})\leq 2r$ and $d(x_1^{(n)},x_{i+1}^{(n)})\in[ir,(i+1)r]$,
\item the sequence $(\inf_{m\neq n}d(S_n,S_m))_{n=1}^\infty$ is strictly positive and tends to infinity as $n$ tends to infinity.
\end{itemize}
Define
\begin{align*}
A &= \bigcup_{n=1}^\infty\{x_1^{(n)},\ldots,x_{m_n-1}^{(n)}\}, \\
B &= \bigcup_{n=1}^\infty\{x_2^{(n)},\ldots,x_{m_n}^{(n)}\}, \\
C &= \bigcup_{n=1}^\infty\{x_1^{(n)},\ldots,x_{m_n}^{(n)}\}.
\end{align*}
Let $p$ be the characteristic function of $A\cup(X\setminus C)$ and let $q$ be the characteristic function of $B\cup(X\setminus C)$. Then $p$ and $q$ are equivalent idempotents. Indeed, let $v\in B(\ell^p(X))$ be defined by
\[ \delta_x\mapsto\begin{cases} \delta_x & x\in X\setminus C, \\ \delta_{x_{i+1}^{(n)}} & x=x_i^{(n)}\;\text{for some $n$ and $i\in\{1,\ldots,m_n-1\}$}, \\ 0 & x=x_{m_n}^{(n)}\;\text{for some $n$}, \end{cases} \]
and let $w\in B(\ell^p(X))$ be defined by
\[ \delta_x\mapsto\begin{cases} \delta_x & x\in X\setminus C, \\ \delta_{x_{i-1}^{(n)}} & x=x_i^{(n)}\;\text{for some $n$ and $i\in\{2,\ldots,m_n\}$}, \\ 0 & x=x_1^{(n)}\;\text{for some $n$}. \end{cases} \]
Then $v,w\in B^p_u(X)$, $wv=p$, and $vw=q$.
Now suppose $B^p_u(X)$ has cancellation. Then by Proposition \ref{canc}, $1-p$ and $1-q$ are equivalent idempotents, say $yz=1-p$ and $zy=1-q$ for some $y,z\in B^p_u(X)$. We may also assume that $y=(1-p)y$ (see \cite[Proposition 4.2.2]{Bl}). 
Note that $1-p$ and $1-q$ are, respectively, the characteristic functions of $\{x_{m_n}^{(n)}:n\in\mathbb{N}\}$ and $\{x_1^{(n)}:n\in\mathbb{N}\}$.
There exists $a\in B^p_u(X)$ with finite propagation $s>0$ such that $||y-a||<1/||z||$.
There also exists $n$ such that $d(x_1^{(n)},E)>s$. Let $e\in B^p_u(X)$ be the characteristic function of $x_1^{(n)}$.
Then $(1-p)ae=0$. 
But then $||ye||=||(1-p)ye||=||(1-p)(y-a)e||\leq||y-a||<1/||z||$ so $||e||=||(1-q)e||=||zye||<1$, which is a contradiction. Hence $B^p_u(X)$ does not have cancellation.
\end{proof}

In the rest of this section, we will restrict our attention to the case where the metric space is a countable, locally finite group equipped with a proper left-invariant metric. Such a metric group is actually a bounded geometry metric space with asymptotic dimension zero. The main result is that the ordered $K_0$ group (with order unit) of the associated $\ell^p$ uniform Roe algebra is a complete invariant for the bijective coarse equivalence class of the underlying countable locally finite group. Along the way, we also compute $(K_0(B^p_u(\Gamma)),K_0(B^p_u(\Gamma))^+,[1]_0)$ for a countable, locally finite group $\Gamma$, showing that it is independent of $p$.

\begin{defn}
A discrete group $\Gamma$ is locally finite if every finitely generated subgroup of $\Gamma$ is finite.
\end{defn}
Every countable discrete group can be equipped with a proper left-invariant metric $d$ that is unique up to bijective coarse equivalence \cite[Lemma
2.1]{MR1871980}. In fact, local finiteness of a countable group $\Gamma$ can be characterized in terms of the asymptotic dimension of the metric space $(\Gamma, d)$ as follows.

\begin{thm} \cite[Theorem 2]{Smith}
Let $\Gamma$ be a countable group equipped with any proper left-invariant metric $d$. Then the following are equivalent:
\begin{enumerate}
\item $\Gamma$ is locally finite.
\item $(\Gamma,d)$ has asymptotic dimension zero.
\end{enumerate}
\end{thm}

We may then apply Theorem \ref{thmAsdim0}, Theorem \ref{thmFinite}, and a couple of straightforward observations to obtain the following result, which is an analog of \cite[Corollary 5.4]{LL} (see also \cite{MR3784048}).

\begin{cor}
Let $\Gamma$ be a countable group equipped with any proper left-invariant metric, and let $p\in[1,\infty)$. Then the following are equivalent:
\begin{enumerate}
\item $\Gamma$ is locally finite.
\item $B^p_u(\Gamma)$ is an inductive limit of subalgebras isometrically isomorphic to $\bigoplus_{k=1}^N M_{d_k}^p$ with norm $||(a_1,\ldots,a_N)||=\max(||a_1||,\ldots,||a_N||)$, where $N,d_1,\ldots,d_k\in\mathbb{N}$, and $M_d^p$ denotes $B(\ell^p(\{1,\ldots,d\}))$.
\item $B^p_u(\Gamma)$ has cancellation.
\item $B^p_u(\Gamma)$ is stably finite.
\item $B^p_u(\Gamma)$ is finite.
\item $B^p_u(\Gamma)$ is a quasidiagonal set in $B(\ell^p(\Gamma))$.
\end{enumerate}
\end{cor}

\begin{proof}
The equivalence of (i), (ii), and (iii) follows from Theorem \ref{thmAsdim0}, while the equivalence of (iv), (v), and (vi) follows from Theorem \ref{thmFinite}. 
The fact that any unital Banach algebra with cancellation will be stably finite can be seen easily using Lemma \ref{finite}(ii) and Proposition \ref{canc}(iii), so we have (iii) $\Rightarrow$ (iv).
Finally, to get (v) $\Rightarrow$ (i), note that if $\Gamma$ is not locally finite, then $\mathbb{Z}$ quasi-isometrically embeds into $\Gamma$ so that $B^p_u(\mathbb{Z})\subset B^p_u(\Gamma)$. However, $B^p_u(\mathbb{Z})$ is not finite so $B^p_u(\Gamma)$ is not finite as well.
\end{proof}

In the following we will compute the ordered $K_0$ group of $B^p_u(\Gamma)$ when $\Gamma$ is a countable locally finite group. Recall that a countable group $\Gamma$ is locally finite if and only if there exists an increasing sequence \[ \{e\}=\Gamma_0\subseteq\Gamma_1\subseteq\Gamma_2\subseteq\cdots \] of finite subgroups of $\Gamma$ such that $\Gamma=\bigcup_{n=0}^\infty\Gamma_n$. Such a sequence of finite subgroups gives rise to a proper left-invariant metric $d$ on $\Gamma$ given by $d(g,h)=\min\{n\in\mathbb{N}:g^{-1}h\in\Gamma_n\}$ for $g,h\in\Gamma$.

The following proposition can be proved in the same way as in the $C^\ast$-algebra case and we omit the proof, referring the reader to \cite[Proposition 4.5]{LL} for details.

\begin{prop} \label{grpRoe}
Let $\Gamma$ be a countable, locally finite group. Consider the triple $(\Gamma,\{\Gamma_n\}_{n=0}^\infty,d)$ as above.
For each $n\in\{0,1,2,\ldots\}$, define $k_n=|\Gamma_n|$ and $r_n=k_{n+1}/k_n$. Then for every $p\in[1,\infty)$,
\[ B^p_u(\Gamma)\cong\varinjlim\left(\prod_{i=1}^\infty M^p_{k_n}(\mathbb{C}),\phi_n\right), \]
where $\phi_n(T_1,T_2,\ldots)=(\mathrm{diag}(T_1,\ldots,T_n),\mathrm{diag}(T_{r_n+1},\ldots,T_{2r_n}),\ldots)$.
\end{prop}

From Proposition \ref{grpRoe} and \cite[Proposition 4.1]{LL}, one sees that the ordered $K_0$ group of $B^p_u(\Gamma)$ is a sequential inductive limit of $\ell^\infty(\mathbb{N},\mathbb{Z})$ when $\Gamma$ is a countable, locally finite group.
Moreover, it is not hard to see that the connecting homomorphism $\phi_n$ ($n=0,1,2,\ldots$) induces the following map on $K_0$ groups:
\[ \alpha_n:\ell^\infty(\mathbb{N},\mathbb{Z})\rightarrow\ell^\infty(\mathbb{N},\mathbb{Z}), \]
\[ \alpha_n((m_1,m_2,\ldots))=(m_1+\cdots+m_{r_n},m_{r_n+1}+\cdots+m_{2r_n},\ldots). \]
Continuity of the $K_0$ functor then gives \[ (K_0(B^p_u(\Gamma)),K_0(B^p_u(\Gamma))^+)\cong\varinjlim(\ell^\infty(\mathbb{N},\mathbb{Z}),\ell^\infty(\mathbb{N},\mathbb{Z})^+,\alpha_n). \]
We now proceed to describe the inductive limit more explicitly. Define
\begin{align*} 
H_\Gamma^{(n)}&=\left\{ (m_1,m_2,\ldots)\in\ell^\infty(\mathbb{N},\mathbb{Z}):\sum_{i=jk_n+1}^{(j+1)k_n}m_i=0\;\text{for all}\;j=0,1,2,\ldots \right\}, \\
H_\Gamma &= \bigcup_{n=0}^\infty H_\Gamma^{(n)}.
\end{align*}
Note that $\{H_\Gamma^{(n)}\}_{n=0}^\infty$ is an increasing sequence of subgroups of $\ell^\infty(\mathbb{N},\mathbb{Z})$ since $k_n$ divides $k_{n+1}$ for each $n$.

\begin{prop} \cite[Proposition 4.6]{LL}
Let $\alpha_n$ and $H_\Gamma$ be as above. Then
\[ \varinjlim(\ell^\infty(\mathbb{N},\mathbb{Z}),\ell^\infty(\mathbb{N},\mathbb{Z})^+,\alpha_n)\cong(\ell^\infty(\mathbb{N},\mathbb{Z})/H_\Gamma,\ell^\infty(\mathbb{N},\mathbb{Z})^+/H_\Gamma), \]
where $\ell^\infty(\mathbb{N},\mathbb{Z})^+/H_\Gamma$ is the collection of all elements in $\ell^\infty(\mathbb{N},\mathbb{Z})/H_\Gamma$ that can be represented by positive sequences.
\end{prop}
Now we can explicitly describe the ordered $K_0$ group of $B^p_u(\Gamma)$ for any countable locally finite group $\Gamma$ (see \cite[Theorem 4.7]{LL} for $p=2$ case):
\begin{thm}\label{k-theory of locally finite}
Let $\Gamma$ be a countable, locally finite group, and let $\{e\}=\Gamma_0\subseteq\Gamma_1\subseteq\Gamma_2\subseteq\cdots$ be an increasing sequence of finite subgroups of $\Gamma$ with $\Gamma=\bigcup_{n=0}^\infty\Gamma_n$. Define $k_n=|\Gamma_n|$,
\begin{align*} 
H_\Gamma^{(n)}&=\left\{ (m_1,m_2,\ldots)\in\ell^\infty(\mathbb{N},\mathbb{Z}):\sum_{i=jk_n+1}^{(j+1)k_n}m_i=0\;\text{for all}\;j=0,1,2,\ldots \right\}, \\
H_\Gamma &= \bigcup_{n=0}^\infty H_\Gamma^{(n)}.
\end{align*}
Then for every $p\in[1,\infty)$, \[ (K_0(B^p_u(\Gamma)),K_0(B^p_u(\Gamma))^+,[1]_0)\cong(\ell^\infty(\mathbb{N},\mathbb{Z})/H_\Gamma,\ell^\infty(\mathbb{N},\mathbb{Z})^+/H_\Gamma,[\mathbf{1}]), \]
where $\mathbf{1}$ denotes the constant sequence with value 1.

In particular, the ordered $K_0$ group of $B^p_u(\Gamma)$ for a countable, locally finite group $\Gamma$ does not depend on $p$.
\end{thm}

\begin{proof}
From the preceding result and remarks prior to that, we now only need to keep track of the order unit $[1]_0$.
The $K_0$ class of the unit of $\prod_{i=1}^\infty\mathbb{C}$ in $\ell^\infty(\mathbb{N},\mathbb{Z})$ is given by the constant 1 sequence. The result then follows since the structure map $\ell^\infty(\mathbb{N},\mathbb{Z})\rightarrow\ell^\infty(\mathbb{N},\mathbb{Z})/H_\Gamma$ for the inductive limit is simply the quotient map (cf. the proof of \cite[Proposition 4.6]{LL}).
\end{proof}
The following corollary generalizes \cite[Corollary 4.9]{LL}:
\begin{cor}\label{ce+lf}
Let $\Gamma$ and $\Lambda$ be countable, locally finite groups with proper left-invariant metrics $d_\Gamma$ and $d_\Lambda$ respectively. The following are equivalent:
\begin{enumerate}
\item[(1)] $(\Gamma,d_\Gamma)$ and $(\Lambda,d_\Lambda)$ are coarsely equivalent. 
\item[(2)] $K_0(B^p_u(\Gamma))\cong K_0(B^p_u(\Lambda))$ for all $p\in[1,\infty)$.
\item[(3)] $K_0(B^p_u(\Gamma))\cong K_0(B^p_u(\Lambda))$ for some $p\in[1,\infty)$.
\end{enumerate}
\end{cor}

\begin{proof}
The implication (1) $\Rightarrow$ (2) is a consequence of \cite[Theorem 3.4]{CL}, while the implication (2) $\Rightarrow$ (3) is obvious, so it remains to show (3) $\Rightarrow$ (1).

Suppose $\Gamma$ and $\Lambda$ are not coarsely equivalent. Then by \cite[Corollary 8]{BZ} it must be the case that one of them is finite while the other is infinite since all countably infinite, locally finite groups are coarsely equivalent to each other. Assume without loss of generality that $\Lambda$ is finite. Then for every $p\in[1,\infty)$, we have $K_0(B^p_u(\Gamma))\cong\ell^\infty(\mathbb{N},\mathbb{Z})/H_\Gamma$ by the previous theorem, and $K_0(B^p_u(\Lambda))\cong\mathbb{Z}$. These two groups are not isomorphic, for instance because $\ell^\infty(\mathbb{N},\mathbb{Z})/H_\Gamma$ is not singly generated.
\end{proof}
Now we recall the notion of supernatural numbers and how it enables us to determine whether two countable, locally finite groups are bijectively coarsely equivalent. We will use this in the proof of Theorem \ref{thm:locfingrp}, which is the main result of this section.

Let $\Gamma$ be a countable, locally finite group, and let $\{p_1,p_2,\ldots\}$ be the set of all prime numbers listed in increasing order. For each $j\in\mathbb{N}$, define
\[ n_j=\sup\{m\in\mathbb{N}:p_j^m\;\text{divides}\;|F|\;\text{for some finite subgroup}\;F\;\text{of}\;\Gamma\}. \]
The sequence $\{n_j\}_{j=1}^\infty$ is called the supernatural number associated to $\Gamma$, which we denote by $s(\Gamma)$.
We usually think of a supernatural number $\{n_j\}_{j=1}^\infty$ as a formal product $\prod_{j=1}^\infty p_j^{n_j}$, so we say that $p_j^m$ divides $s(\Gamma)$ if $m\leq n_j$. We say that two supernatural numbers are equal if they are equal as sequences.

\begin{thm} (\cite[Theorem 5]{Protasov} and \cite[Theorem 3.10]{LL})\label{super}
Let $\Gamma$ and $\Lambda$ be countable, locally finite groups with proper left-invariant metrics $d_\Gamma$ and $d_\Lambda$ respectively. Then the following are equivalent:
\begin{enumerate}
\item $(\Gamma,d_\Gamma)$ and $(\Lambda,d_\Lambda)$ are bijectively coarsely equivalent.
\item $\Gamma$ and $\Lambda$ have the same supernatural number, i.e., $s(\Gamma)=s(\Lambda)$.
\end{enumerate}
\end{thm}

Finally, we come to the main result of this section, which generalizes \cite[Theorem 4.10]{LL}.

\begin{thm} \label{thm:locfingrp}
Let $\Gamma$ and $\Lambda$ be countable, locally finite groups with proper left-invariant metrics $d_\Gamma$ and $d_\Lambda$ respectively. Then the following conditions are equivalent:
\begin{enumerate}
\item[(1)] $(\Gamma,d_\Gamma)$ and $(\Lambda,d_\Lambda)$ are bijectively coarsely equivalent.
\item[(2)] For every $p\in[1,\infty)$, there is an isometric isomorphism $\phi:B^p_u(\Gamma)\rightarrow B^p_u(\Lambda)$ such that $\phi(\ell^\infty(\Gamma))=\ell^\infty(\Lambda)$.
\item[(3)] $B^p_u(\Gamma)$ and $B^p_u(\Lambda)$ are isometrically isomorphic for every $p\in[1,\infty)$.
\item[(4)] $B^p_u(\Gamma)$ and $B^p_u(\Lambda)$ are isometrically isomorphic for some $p\in[1,\infty)$.
\item[(5)] For some $p\in[1,\infty)$, there is an isometric isomorphism $\phi:B^p_u(\Gamma)\rightarrow B^p_u(\Lambda)$ such that $\phi(\ell^\infty(\Gamma))=\ell^\infty(\Lambda)$.
\item[(6)] $(K_0(B^p_u(\Gamma)),K_0(B^p_u(\Gamma))^+,[1]_0)\cong(K_0(B^p_u(\Lambda)),K_0(B^p_u(\Lambda))^+,[1]_0)$ for every $p\in[1,\infty)$.
\item[(7)] $(K_0(B^p_u(\Gamma)),K_0(B^p_u(\Gamma))^+,[1]_0)\cong(K_0(B^p_u(\Lambda)),K_0(B^p_u(\Lambda))^+,[1]_0)$ for some $p\in[1,\infty)$.
\item[(8)] $(K_0(B^p_u(\Gamma)),[1]_0)\cong(K_0(B^p_u(\Lambda)),[1]_0)$ for every $p\in[1,\infty)$.
\item[(9)] $(K_0(B^p_u(\Gamma)),[1]_0)\cong(K_0(B^p_u(\Lambda)),[1]_0)$ for some $p\in[1,\infty)$.
\end{enumerate}
\end{thm}

\begin{proof}
The equivalences of statements (1) through (5) follow from \cite[Theorem 2.6]{CL}, \cite[Theorem 4.10]{LL} and \cite[Lemma 8]{MR0043392}.
It is obvious that (3) $\Rightarrow$ (6) $\Rightarrow$ (8) $\Rightarrow$ (9), and that (6) $\Rightarrow$ (7) $\Rightarrow$ (9). Thus, it remains to show that (9) $\Rightarrow$ (1).

The proof of (9) $\Rightarrow$ (1) is identical to the proof of \cite[Theorem 4.10]{LL}, but we include the details here for the convenience of the reader. For simplicity, we write $A$ for $B^p_u(\Gamma)$ and $B$ for $B^p_u(\Lambda)$.
Let $\phi:(K_0(A),[1_A]_0)\rightarrow(K_0(B),[1_B]_0)$ be an isomorphism. Suppose, for the purpose of contradiction, that $(\Gamma,d_\Gamma)$ and $(\Lambda,d_\Lambda)$ are not bijectively coarsely equivalent. Then the associated supernatural numbers $s(\Gamma)$ and $s(\Lambda)$ are different by Theorem \ref{super}. Without loss of generality, we may assume that there is a prime number $p$ and a positive integer $r$ such that $p^r$ divides $s(\Gamma)$ but not $s(\Lambda)$.
Let $\{\Gamma_n\}_{n=0}^\infty$ and $\{\Lambda_n\}_{n=0}^\infty$ be increasing sequences of finite subgroups of $\Gamma$ and $\Lambda$ respectively, with $\Gamma=\bigcup_{n=0}^\infty\Gamma_n$ and $\Lambda=\bigcup_{n=0}^\infty\Lambda_n$.
Define $k_n=|\Gamma_n|$, $H_\Gamma^{(n)}$, and $H_\Gamma$ as in Theorem \ref{k-theory of locally finite}. Similarly, define $k_n'=|\Lambda_n|$, $H_\Lambda^{(n)}$ and $H_\Lambda$.

Since $p^r$ divides $s(\Gamma)$, it divides $k_n$ for some $n$. Let $q$ be any idempotent in $\prod_{i=1}^\infty M_{k_n}(\mathbb{C})$ with pointwise rank $k_n/p^r$. Then $p^r([q]_0)=[1_A]_0$ in $K_0(A)$. Let $[q']_0=\phi([q]_0)\in K_0(B)$. Then $p^r([q']_0)=[1_B]_0$. Write $[q']_0=[(m_1,m_2,\ldots)]\in\ell^\infty(\mathbb{N},\mathbb{Z})/H_\Lambda$. Then $(p^rm_1-1,p^rm_2-1,\ldots)\in H_\Lambda^{(n)}$ for some positive integer $n$. In particular, $p^r(m_1+m_2+\cdots+m_{k_n'})-k_n'=0$, which is impossible because $p^r$ does not divide $k_n'$ for it would divide $s(\Lambda)$ otherwise, so we have reached a contradiction.
\end{proof}

\ \newline
{\bf Acknowledgments}.
We would like to thank Hung-Chang Liao, N. Christopher Phillips and Rufus Willett for helpful and enlightening discussions.


\bibliographystyle{plain}
\bibliography{mybib}

\begin{thebibliography}{10}

\bibitem{amenUR}
P.~Ara, K.~Li, F.~Lled\'{o}, and J.~Wu.
\newblock Amenability and uniform {R}oe algebras.
\newblock {\em J. Math. Anal. Appl.}, 459:686--716, 2018.

\bibitem{AmenCoa}
P.~Ara, K.~Li, F.~Lled\'{o}, and J.~Wu.
\newblock Amenability of coarse spaces and $\mathbb{K}$-algebras.
\newblock {\em Bull. Math. Sci.}, 8(2):257–306, 2018.

\bibitem{MR2645049}
T.~Banakh, J.~Higes, and I.~Zarichnyi.
\newblock The coarse classification of countable abelian groups.
\newblock {\em Trans. Amer. Math. Soc.}, 362(9):4755--4780, 2010.

\bibitem{BZ}
T.~Banakh and I.~Zarichnyi.
\newblock Characterizing the {C}antor bi-cube in asymptotic categories.
\newblock {\em Groups Geom. Dyn.}, 5(4):691--728, 2011.

\bibitem{Bl}
B.~Blackadar.
\newblock {\em {$K$}-theory for operator algebras}, volume~5 of {\em MSRI
  Publications}.
\newblock Cambridge University Press, 2nd edition, 1998.

\bibitem{phible}
D.P. Blecher and N.C. Phillips.
\newblock {$L^p$} operator algebras with approximate identities {I}.
\newblock {\em preprint}, 2018.
\newblock arXiv:1802.04424.

\bibitem{MR2363428}
J.~Brodzki, G.A. Niblo, and N.J. Wright.
\newblock Property {A}, partial translation structures, and uniform embeddings
  in groups.
\newblock {\em J. Lond. Math. Soc. (2)}, 76(2):479--497, 2007.

\bibitem{BO}
N.P. Brown and N.~Ozawa.
\newblock {\em {$C^*$}-algebras and finite-dimensional approximations},
  volume~88 of {\em Graduate Studies in Mathematics}.
\newblock American Mathematical Society, Providence, RI, 2008.

\bibitem{Chung2}
Y.C. Chung.
\newblock Dynamical complexity and {$K$}-theory of {$L^p$} operator crossed
  products.
\newblock {\em preprint}, 2018.
\newblock arXiv:1611.09000.

\bibitem{Chung1}
Y.C. Chung.
\newblock Quantitative {$K$}-theory for {B}anach algebras.
\newblock {\em J. Funct. Anal.}, 274(1):278--340, 2018.

\bibitem{CL}
Y.C. Chung and K.~Li.
\newblock Rigidity of {$\ell^p$} {R}oe-type algebras.
\newblock {\em Bull. Lond. Math. Soc.}, 50(6):1056--1070, 2018.

\bibitem{Cuntz77}
J.~Cuntz.
\newblock Simple {$C^\ast$}-algebras generated by isometries.
\newblock {\em Comm. Math. Phys.}, 57(2):173--185, 1977.

\bibitem{Cuntz78}
J.~Cuntz.
\newblock Murray-von {N}eumann equivalence of projections in infinite simple
  {$C^\ast$}-algebras.
\newblock {\em Rev. Roumaine Math. Pures Appl.}, 23:1011--1014, 1978.

\bibitem{Cuntz81}
J.~Cuntz.
\newblock {$K$}-theory for certain {$C^\ast$}-algebras.
\newblock {\em Ann. of Math.}, 113(1):181--197, 1981.

\bibitem{Dales}
H.G. Dales.
\newblock {\em Banach algebras and automatic continuity}, volume~24 of {\em
  London Mathematical Society Student Monographs}.
\newblock Oxford University Press, 2001.

\bibitem{DLR}
H.G. Dales, N.J. Laustsen, and C.J. Read.
\newblock A properly infinite {B}anach $\ast$-algebra with a non-zero, bounded
  trace.
\newblock {\em Studia Math.}, 155(2):107--129, 2003.

\bibitem{CornHarpe}
Y.~de~Cornulier and P.~de~la Harpe.
\newblock {\em Metric geometry of locally compact groups}, volume~25 of {\em
  EMS Tracts in Mathematics}.
\newblock European Mathematical Society, 2016.

\bibitem{MR1876896}
E.~Guentner and J.~Kaminker.
\newblock Exactness and the {N}ovikov conjecture.
\newblock {\em Topology}, 41(2):411--418, 2002.

\bibitem{MR3451966}
R.~Hagger, M.~Lindner, and M.~Seidel.
\newblock Essential pseudospectra and essential norms of band-dominated
  operators.
\newblock {\em J. Math. Anal. Appl.}, 437(1):255--291, 2016.

\bibitem{MR0270173}
P.R. Halmos.
\newblock Ten problems in {H}ilbert space.
\newblock {\em Bull. Amer. Math. Soc.}, 76:887--933, 1970.

\bibitem{MR1739727}
N.~Higson and J.~Roe.
\newblock Amenable group actions and the {N}ovikov conjecture.
\newblock {\em J. Reine Angew. Math.}, 519:143--153, 2000.

\bibitem{MR0043392}
R.V. Kadison.
\newblock Isometries of operator algebras.
\newblock {\em Ann. Of Math. (2)}, 54:325--338, 1951.

\bibitem{MR3158244}
J.~Kellerhals, N.~Monod, and M.~R{\o}rdam.
\newblock Non-supramenable groups acting on locally compact spaces.
\newblock {\em Doc. Math.}, 18:1597--1626, 2013.

\bibitem{Laustsen}
N.J. Laustsen.
\newblock On ring-theoretic (in)finiteness of {B}anach algebras of operators on
  {B}anach spaces.
\newblock {\em Glasgow Math. J.}, 45:11--19, 2003.

\bibitem{Leavitt62}
W.G. Leavitt.
\newblock The module type of a ring.
\newblock {\em Trans. Amer. Math. Soc.}, 103(1):113--130, 1962.

\bibitem{Leavitt65}
W.G. Leavitt.
\newblock The module type of homomorphic images.
\newblock {\em Duke Math. J.}, 32(2):305--311, 1965.

\bibitem{LL}
K.~Li and H.-C. Liao.
\newblock Classification of uniform {R}oe algebras of locally finite groups.
\newblock {\em J. Operator Theory}, 80(1):25--46, 2018.

\bibitem{MR3926163}
K.~Li, Z.~Wang, and J.~Zhang.
\newblock A quasi-local characterisation of {$L^p$}-{R}oe algebras.
\newblock {\em J. Math. Anal. Appl.}, 474(2):1213--1237, 2019.

\bibitem{lowdimUR}
K.~Li and R.~Willett.
\newblock Low dimensional properties of uniform {R}oe algebras.
\newblock {\em J. Lond. Math. Soc. (2)}, 97(1):98--124, 2018.

\bibitem{MR3212726}
M.~Lindner and M.~Seidel.
\newblock An affirmative answer to a core issue on limit operators.
\newblock {\em J. Funct. Anal.}, 267(3):901--917, 2014.

\bibitem{MatuiRordam}
H.~Matui and M.~R{\o}rdam.
\newblock Universal properties of group actions on locally compact spaces.
\newblock {\em J. Funct. Anal.}, 268(12):3601--3648, 2015.

\bibitem{MR1763912}
N.~Ozawa.
\newblock Amenable actions and exactness for discrete groups.
\newblock {\em C. R. Acad. Sci. Paris S\'er. I Math.}, 330(8):691--695, 2000.

\bibitem{PhilOpen}
N.C. Phillips.
\newblock Open problems related to operator algebras on {$L^p$} spaces.
\newblock
  https://www.math.ksu.edu/events/conference/gpots2014/LpOpAlgQuestions.pdf.

\bibitem{Phil12}
N.C. Phillips.
\newblock Analogs of {C}untz algebras on {$L^p$} spaces.
\newblock {\em preprint}, 2012.
\newblock arXiv:1309.4196.

\bibitem{Phil13}
N.C. Phillips.
\newblock Crossed products of {$L^p$} operator algebras and the {$K$}-theory of
  {C}untz algebras on {$L^p$} spaces.
\newblock {\em preprint}, 2013.
\newblock arXiv:1309.6406.

\bibitem{PhilSimplicity}
N.C. Phillips.
\newblock Simplicity of {UHF} and {C}untz algebras on {$L^p$} spaces.
\newblock {\em preprint}, 2013.
\newblock arXiv:1309.0115.

\bibitem{PhilViola}
N.C. Phillips and M.G. Viola.
\newblock Classification of {$L^p$} {AF} algebras.
\newblock {\em preprint}, 2017.
\newblock arXiv:1707.09257.

\bibitem{Pop}
C.~Pop.
\newblock Finite sums of commutators.
\newblock {\em Proc. Amer. Math. Soc.}, 130(10):3039--3041, 2002.

\bibitem{Protasov}
I.V. Protasov.
\newblock Morphisms of ball's structures of groups and graphs.
\newblock {\em Ukrainian Math. J.}, 54(6):1027--1037, 2002.

\bibitem{MR2873171}
M.~R{\o}rdam and A.~Sierakowski.
\newblock Purely infinite {$C^*$}-algebras arising from crossed products.
\newblock {\em Ergodic Theory Dynam. Systems}, 32(1):273--293, 2012.

\bibitem{MR3784048}
E.~Scarparo.
\newblock Characterizations of locally finite actions of groups on sets.
\newblock {\em Glasg. Math. J.}, 60(2):285--288, 2018.

\bibitem{MR3151282}
M.~Seidel.
\newblock Fredholm theory for band-dominated and related operators: a survey.
\newblock {\em Linear Algebra Appl.}, 445:373--394, 2014.

\bibitem{MR1905840}
G.~Skandalis, J.-L. Tu, and G.~Yu.
\newblock The coarse {B}aum-{C}onnes conjecture and groupoids.
\newblock {\em Topology}, 41(4):807--834, 2002.

\bibitem{Smith}
J.~Smith.
\newblock On asymptotic dimension of countable abelian groups.
\newblock {\em Topology Appl.}, 153(12):2047--2054, 2006.

\bibitem{vspakula2017metric}
J.~{\v{S}}pakula and R.~Willett.
\newblock A metric approach to limit operators.
\newblock {\em Trans. Amer. Math. Soc.}, 369(1):263--308, 2017.

\bibitem{vspakula2018quasi}
J.~{\v{S}}pakula and J.~Zhang.
\newblock Quasi-locality and property {A}.
\newblock {\em preprint}, 2018.
\newblock arXiv:1809.00532.

\bibitem{MR1871980}
J.-L. Tu.
\newblock Remarks on {Y}u's ``property {A}'' for discrete metric spaces and
  groups.
\newblock {\em Bull. Soc. Math. France}, 129(1):115--139, 2001.

\bibitem{Wei}
S.~Wei.
\newblock On the quasidiagonality of {R}oe algebras.
\newblock {\em Sci. China Math.}, 54(5):1011--1018, 2011.

\bibitem{WZ}
W.~Winter and J.~Zacharias.
\newblock The nuclear dimension of {$C^\ast$}-algebras.
\newblock {\em Adv. Math.}, 224:461--498, 2010.

\bibitem{Zhang18}
J.~Zhang.
\newblock Extreme cases of limit operator theory on metric spaces.
\newblock {\em Integral Equations and Operator Theory}, 90(6):73, 2018.

\end{thebibliography}

\end{document}